\documentclass{article}%
\usepackage{amsmath}
\usepackage{amssymb}
\usepackage{amsfonts}
\usepackage{graphicx}%
\providecommand{\U}[1]{\protect\rule{.1in}{.1in}}
\newtheorem{theorem}{Theorem}

\newtheorem{corollary}[theorem]{Corollary}

\newtheorem{lemma}[theorem]{Lemma}

\newtheorem{proposition}[theorem]{Proposition}

\newenvironment{proof}[1][Proof]{\noindent\textbf{#1.} }{\ \rule{0.5em}{0.5em}}
\begin{document}

\title{The existence of stable BGK waves}
\author{Yan Guo\\Division of Applied Mathematics\\Brown University\\Providence, RI 02912, USA
\and Zhiwu Lin\\School of Mathematics\\Georgia Institute of Technology\\Atlanta, GA 30332, USA}
\date{}
\maketitle

\begin{abstract}
The $1$D Vlasov-Poisson system is the simplest kinetic model for describing an
electrostatic collisonless plasma, and the BGK waves are its famous exact
steady solutions. They play an important role on the long time dynamics of a
collisionless plasma as potential "final states" or "attractors", thanks to
many numerical simulations and observations. Despite their importance, the
existence of stable BGK waves has been an open problem since their discovery
in 1957. In this paper, linearly stable BGK waves are constructed near
homogeneous states.

\end{abstract}

\section{Introduction}

The 1D Vlasov-Poisson (VP) system is the simplest kinetic model for describing
a collisionless plasma:
\begin{equation}
\partial_{t}f_{\pm}+v\partial_{x}f_{\pm}\pm E\partial_{v}f_{\pm}=0,\text{
\ \ \ }E_{x}=\int[f_{+}-f_{-}]dv
\end{equation}
where $f_{\pm}(t,x,v)$ denote distribution functions for ions $(+)$ and
electrons $(-)$ respectively, with their self-consistent electric field $E.$
In plasma physics literature, 1D VP for the electrons with a fixed ion
background is often studied. The equation becomes
\begin{subequations}
\begin{equation}
\partial_{t}f+v\partial_{x}f-E\partial_{v}f=0,\ \ \ \ \ \ \ \ E_{x}%
=-\int_{-\infty}^{+\infty}f\ dv+1,
\end{equation}
where $1$ is the fixed ion density. There are important physical phenomena
from the study of Vlasov models. They include Landau damping (\cite{villani09}
\cite{bmm} \cite{lin-zeng-BGK}), and kinetic instabilities such as two-stream
instabilities (\cite{am67}). The BGK waves were discovered by
Bernstein-Greene-Kruskal in 1957 (\cite{bgk}), as exact, spatially periodic
steady state solutions to the Vlasov-Poisson system. Consider general BGK
waves with non-even distribution functions for ions $(+)$ and elctrons
$\left(  -\right)  $
\end{subequations}
\begin{equation}
f_{\pm}^{\beta}\left(  x,v\right)  =\left\{
\begin{array}
[c]{cc}%
\mu_{\pm,+}\left(  e_{\pm}\right)   & \text{when\ }v>0\\
\mu_{\pm,-}\left(  e_{\pm}\right)   & \text{when }v<0
\end{array}
\right.  ,\label{defn-f-pn-BGK}%
\end{equation}
where $e_{\pm}=\frac{v^{2}}{2}\pm\beta$. The self-consistent potential $\beta$
satisfies
\begin{align}
-\beta_{xx} &  =\int_{v>0}\mu_{+,+}(e_{+})dv+\int_{v<0}\mu_{+,-}%
(e_{+})dv\label{bgk-uneven}\\
&  -\int_{v>0}\mu_{-,+}(e_{-})dv-\int_{v<0}\mu_{-,-}(e_{-})dv\equiv h\left(
\beta\right)  .\nonumber
\end{align}
Denote the homogeneous state
\begin{equation}
f_{0,\pm}\left(  v\right)  =\left\{
\begin{array}
[c]{cc}%
\mu_{\pm,+}\left(  \frac{1}{2}v^{2}\right)   & \text{when\ }v>0\\
\mu_{\pm,-}\left(  \frac{1}{2}v^{2}\right)   & \text{when }v<0
\end{array}
\right.  .\label{defn-f-0-pn-uneven}%
\end{equation}
For the case with fixed ion background, let the electron distribution be%
\[
f^{\beta}\left(  x,v\right)  =\left\{
\begin{array}
[c]{cc}%
\mu_{-,+}\left(  e_{-}\right)   & \text{when\ }v>0\\
\mu_{-,-}\left(  e_{-}\right)   & \text{when }v<0
\end{array}
\right.  ,
\]
and
\begin{equation}
-\beta_{xx}=1-\int_{v>0}\mu_{-,+}(e_{-})dv-\int_{v<0}\mu_{-,-}(e_{-})dv\equiv
h\left(  \beta\right)  ,\label{bgk uneven fixed ion}%
\end{equation}
where $1$ is the ion density. Denote
\begin{equation}
f_{0}\left(  v\right)  =\left\{
\begin{array}
[c]{cc}%
\mu_{-,+}\left(  \frac{1}{2}v^{2}\right)   & \text{when\ }v>0\\
\mu_{-,-}\left(  \frac{1}{2}v^{2}\right)   & \text{when }v<0
\end{array}
\right.  .\label{defn-f-0-fixed-uneven}%
\end{equation}
The existence of small BGK waves satisfying (\ref{bgk-uneven})
(\ref{bgk uneven fixed ion}) follows from the Liapunov center theorem (see
e.g. \cite{gs2}).

\begin{lemma}
Consider the ODE $-\beta_{xx}=h\left(  \beta\right)  $ where $h\in C^{1}$. If
$h\left(  0\right)  =0,\ h^{\prime}\left(  0\right)  =\left(  \frac{2\pi
}{P_{0}}\right)  ^{2}>0$, then there exist a family of periodic solutions with
$||\beta||_{\infty}=\varepsilon<<1$ with minimum period $P_{\beta}\rightarrow
P_{0}$ when $\varepsilon\rightarrow0$. We can normalize $\beta(x)$ to be even
in $\left[  -\frac{P_{\beta}}{2},\frac{P_{\beta}}{2}\right]  $ with maximum at
$x=0,\ $minimum at $x=\pm\frac{P_{\beta}}{2},\ \min\beta=-\max\beta$, so that
\begin{equation}
\beta(x)=\varepsilon\cos\frac{2\pi}{P_{\beta}}x+O(\varepsilon^{2}).
\label{leading-order-beta}%
\end{equation}

\end{lemma}

Remarkably, these BGK waves play crucial roles in understanding long time
dynamics of the Vlasov-Poisson system, which have been an important topic in
plasma physics. Many numerical simulations \cite{demeio-holloway}
\cite{manfredi97} \cite{rosenbluth-et-98} \cite{driscoll-et-04}
\cite{brunetti-et-00} \cite{schamel-review05} \cite{schamel12} indicate that
for initial data near a stable homogeneous state including Maxwellian, the
asymptotic behavior of approaching a BGK wave or superposition of BGK waves is
usually observed. Moreover, BGK waves also appear as the `attractor' or "final
states" for the saturation of an unstable homogeneous state (\cite{am67}
\cite{Cheng76} \cite{deimo-zweifel} \cite{driscoll-et-04} \cite{gibs88}
\cite{Morrison-et}). For example, in \cite{gibs88}, starting near a BGK wave
with double period which is unstable, the authors observed the gradual
evolution to another BGK wave of minimal period. To understand such long time
behaviors, an important first step is to construct stable BGK waves.

Ever since the discovery of BGK waves, their stability has been an active
research area. There has been a lot of formal analysis in the physical
literature. Instability of BGK waves to perturbations of multiple periods was
proved in \cite{gs2} for waves of small amplitude and in \cite{lin01}
\cite{lin-cpam} for waves of large amplitude. Unfortunately, despite intense
efforts, no stable BGK waves to perturbations of minimum period have been
found since 1957.

One difficulty is that stable BGK waves cannot be obtained by the traditional
energy-Casimir method, which was first used by Newcomb in 1950s
(\cite{trievelpiece}) to prove nonlinear stability of Maxwellian. This method
requires the profiles $\mu_{\pm,\pm}\ $to be monotone decreasing to $e_{\pm}$,
which implies that $h\left(  \beta\right)  $ defined in (\ref{bgk-uneven}) or
(\ref{bgk uneven fixed ion}) is a decreasing function of $\beta$. So, by
differentiating (\ref{bgk-uneven}) or (\ref{bgk uneven fixed ion}) and
integrating it with $\beta_{x}$, we get
\[
\int\left[  \left(  \beta_{xx}\right)  ^{2}-h^{\prime}\left(  \beta\right)
\left(  \beta_{x}\right)  ^{2}\right]  dx=0,
\]
and thus $\beta_{x}\equiv0$ (i.e. homogeneous states). So for any nontrivial
BGK waves, the profiles $\mu_{\pm,\pm}$ cannot be monotone and thus the
energy-Casimir method does not work. For the homogeneous equilibria, due to
the separation of Fourier modes, a simple dispersion relation function can be
analyzed to get the Penrose stability criterion (\cite{penrose}). However,
even for small BGK waves, due to the coupling of infinitely many modes, the
dispersion relation is difficult to study for linear stability.

In the rest of this paper, we assume $\mu_{\pm,+}^{\prime}(\theta),\ \mu
_{\pm,-}^{\prime}(\theta)\equiv0$ for $|\theta|\leq\sigma_{\pm}$ and denote
$\sigma=\min\left\{  \sigma_{+},\sigma_{-}\right\}  .$That is, the
distribution function is assumed to be flat near $0$. First, this simplifies
some technical steps in our construction. Second, this assumption is also
physically relevant. It was known that in the long time evolution of VP near a
homogeneous state, the distribution function can develop a plateau due to the
resonant particles (\cite{bellan}).

Our first result shows that small BGK waves with non-even distribution are
generally spectrally stable, that is, the spectra of the linearized VP
operator lies in the imaginary axis.

\begin{theorem}
\label{thm-main-uneven}

(i) (Uneven and Two species) Assume $\mu_{\pm,\pm}\in C^{3}\left(
\mathbf{R}\right)  $ are nonnegative and
\begin{equation}
\max_{1\leq i\leq3}|\mu_{\pm,\pm}^{(i)}(y)|\leq C(1+|y|)^{-\gamma}%
\ (\gamma>1),\ \mu_{\pm,+}^{\prime}(\theta),\ \mu_{\pm,-}^{\prime}%
(\theta)\equiv0\text{ for }|\theta|\leq\sigma_{\pm}.\label{condition-decay}%
\end{equation}
Define $f_{0,\pm}\left(  v\right)  $ by (\ref{defn-f-0-pn-uneven}) and
assume:
\begin{equation}
\int f_{0,+}\left(  v\right)  dv=\int f_{0,-}\left(  v\right)  dv,\ \int%
\frac{f_{0,+}^{\prime}\left(  v\right)  +f_{0,-}^{\prime}\left(  v\right)
}{v}=\left(  \frac{2\pi}{P_{0}}\right)  ^{2},\label{condition-bifurcation}%
\end{equation}%
\begin{equation}
\int\frac{f_{0,+}^{\prime}\left(  v\right)  +f_{0,-}^{\prime}\left(  v\right)
}{v^{2}}\neq0,\label{condition-nondegenerate}%
\end{equation}%
\begin{equation}
\int\frac{f_{0,+}^{\prime}\left(  v\right)  +f_{0,-}^{\prime}\left(  v\right)
}{v-v_{r}}dv<\left(  \frac{2\pi}{P_{0}}\right)  ^{2}%
,\ \label{condition-penrose}%
\end{equation}
for any critical point of $v_{r}$ of $f_{0,+}\left(  v\right)  +f_{0,-}\left(
v\right)  $ with $\left\vert v_{r}\right\vert \geq\sigma$, and
\begin{equation}
\int\frac{f_{0,+}^{\prime}\left(  v\right)  +f_{0,-}^{\prime}\left(  v\right)
}{v-v_{r}}dv\neq\left(  \frac{2\pi}{P_{0}}\right)  ^{2}\text{, when
}\left\vert v_{r}\right\vert \leq\sigma\text{.}\label{condition-unequal-homo}%
\end{equation}
Then when $\varepsilon=||\beta||_{\infty}$ is small enough, the BGK wave
$\left[  f_{\pm}^{\beta},-\beta_{x}\right]  \ $satisfying (\ref{bgk-uneven})
is spectrally stable against $P_{\beta}$-periodic perturbations.

(ii) (Uneven and Fixed ion background) Assume $\mu_{-,\pm}\in C^{3}\left(
\mathbf{R}\right)  $ are nonnegative and
\begin{equation}
\max_{1\leq i\leq3}|\left(  \mu_{-,\pm}\right)  ^{(i)}(y)|\leq(1+|y|)^{-\gamma
}(\gamma>1),\ \left(  \mu_{-,\pm}\right)  ^{\prime}(\theta)\equiv0\text{ for
}|\theta|\leq\sigma_{\pm}. \label{condition-decay-+-}%
\end{equation}
Define $f_{0}\left(  v\right)  $ by (\ref{defn-f-0-fixed-uneven}) and assume:
$\int\frac{f_{0}^{\prime}\left(  v\right)  }{v^{2}}\neq0,$ $\ $
\begin{equation}
\int f_{0}\left(  v\right)  dv=1,\ \int\frac{f_{0}^{\prime}\left(  v\right)
}{v}=\left(  \frac{2\pi}{P_{0}}\right)  ^{2}%
,\ \ \label{condition-bifurcation-fixed ion}%
\end{equation}%
\begin{equation}
\int\frac{f_{0}^{\prime}\left(  v\right)  }{v-v_{r}}dv<\left(  \frac{2\pi
}{P_{0}}\right)  ^{2} \label{condition-penrose-fixed ion}%
\end{equation}
$\ $for any critical point of $v_{r}$ of $f_{0}\left(  v\right)  $ with
$\left\vert v_{r}\right\vert \geq\sigma$, and \
\begin{equation}
\int\frac{f_{0}^{\prime}\left(  v\right)  }{v-v_{r}}dv\neq\left(  \frac{2\pi
}{P_{0}}\right)  ^{2},\ \text{when\ }\left\vert v_{r}\right\vert \leq
\sigma\text{. } \label{condition-unequal-homo-fixed}%
\end{equation}
Then small BGK waves $\left[  f^{\beta},-\beta_{x}\right]  \ $satisfying
(\ref{bgk uneven fixed ion}) are spectrally stable against $P_{\beta}%
$-periodic perturbations.
\end{theorem}

The conditions in the above Theorem are quite natural and general:
(\ref{condition-bifurcation}) is the bifurcation condition of small BGK waves;
(\ref{condition-nondegenerate}) is a non-degeneracy condition which is true
for generic non-even profiles; (\ref{condition-penrose}) is the Penrose
stability condition (at period $P_{0}$) for the flat homogeneous states
$f_{0,\pm}$. The condition (\ref{condition-unequal-homo}) is to ensure that
$0$ is the only discrete eigenvalue for the homogeneous profile with period
$P_{0}$. For the fixed ion case, the conditions are similar with $f_{0,\pm}$
being replaced by $f_{0}$. We refer to the final section for more explicit
construction of examples. Theorem \ref{thm-main-uneven} shows that for general
non-even Penrose stable profiles flat near $0$, the small BGK waves are
linearly stable.

Next, we give a sharp stability criterion for small BGK waves with even profiles.

\begin{theorem}
\label{thm-main-even}

(i) (Even and Two species) Assume $\mu_{\pm,-}=\mu_{\pm,+}=\mu_{\pm}\in
C^{3}\left(  \mathbf{R}\right)  $ are nonnegative and
\[
\max_{1\leq i\leq3}|\mu_{\pm}^{(i)}(y)|\leq(1+|y|)^{-\gamma}(\gamma
>1),\ \mu_{\pm}^{\prime}(\theta)\equiv0\text{ for }|\theta|\leq\sigma_{\pm}.
\]
Denote $f_{0,\pm}\left(  v\right)  =\mu_{\pm}\left(  \frac{1}{2}v^{2}\right)
$ and assume: (\ref{condition-bifurcation}), (\ref{condition-penrose}),
(\ref{condition-unequal-homo}) and
\begin{equation}
\int\frac{[\mu_{+}^{\prime}+\mu_{-}^{\prime}]\left(  \frac{1}{2}v^{2}\right)
}{v^{2}}dv>0.\label{condition-positive}%
\end{equation}
$\ $ Then the small BGK waves $[\mu_{\pm}(e_{\pm}),-\beta_{x}]$ satisfying
(\ref{bgk-uneven}) are spectrally stable if
\[
\int v^{-2}[\mu_{+}^{^{\prime}}(\frac{v^{2}}{2})-\mu_{-}^{\prime}(\frac{v^{2}%
}{2})]>0\,\ (\text{equivalently }P_{\beta}^{\prime}<0),
\]
and unstable if
\[
\int v^{-2}[\mu_{+}^{^{\prime}}(\frac{v^{2}}{2})-\mu_{-}^{\prime}(\frac{v^{2}%
}{2})]<0\ \ (\text{equivalently }P_{\beta}^{\prime}>0).
\]
Here, the derivative $P_{\beta}^{\prime}\ $is respect to $\varepsilon
=\max|\beta|$. Moreover, for the stable case, there exists a pair of nonzero
imaginary eigenvalues of the linearized VP operator around $[\mu_{\pm}(e_{\pm
}),-\beta_{x}]$.

(ii) (Even and Fixed ion background) Assume $\mu_{-,+}=\mu_{-,-}=\mu\in
C^{3}\left(  \mathbf{R}\right)  $ is nonnegative and
\[
\max_{1\leq i\leq3}|\mu^{(i)}(y)|\leq(1+|y|)^{-\gamma}(\gamma>1),\ \mu
^{\prime}(\theta)\equiv0\text{ for }|\theta|\leq\sigma.
\]
Denote $f_{0}\left(  v\right)  =\mu\left(  \frac{1}{2}v^{2}\right)  $ and
assume: (\ref{condition-bifurcation-fixed ion}),
(\ref{condition-penrose-fixed ion}), (\ref{condition-unequal-homo-fixed}) and
\begin{equation}
\int\frac{\mu^{\prime}\left(  \frac{1}{2}v^{2}\right)  }{v^{2}}dv>0.
\label{condition-positive-fixed ion}%
\end{equation}
Then the small BGK waves $[\mu(e_{-}),-\beta_{x}]$ satisfying
(\ref{bgk uneven fixed ion}) are unstable.
\end{theorem}

Let $\sigma_{\pm}$ be the width of flatness near $0$ for $\mu_{\pm}$. It is
shown in Section 6 that when $\sigma_{\pm}$ is small,
(\ref{condition-positive}) is always satisfied. Moreover, when $\sigma_{-}%
\gg\sigma_{+}$ $(\sigma_{-}\ll\sigma_{+})$, the stability (instability)
condition $P_{\beta}^{\prime}<0$ $\left(  P_{\beta}^{\prime}<0\right)  $ is
satisfied. So for the even and two-species case, we can construct both stable
and unstable small BGK waves. For the stable case, there exist a pair of
purely imaginary nonzero eigenvalues of the linearized VP operator. This
suggests that there exists a time periodic solution of the linearized VP
equation, for which the electric field does not decay in time. So there is no
Landau damping even at the linear level. In contrast, for linearly stable BGK
waves with uneven profiles as in Theorem \ref{thm-main-uneven} (i), such
nonzero purely imaginary eigenvalues do not exist and the Landau damping might
be true. The problems of constructing nonlinear time periodic solutions for
the even case and proving Landau damping for the uneven case are currently
under investigation.

Last, we discuss some key ideas in the proof. Our analysis relies on a
delicate perturbation argument from a stable homogeneous equilibria. It is
well-known that the original spectra analysis around small BGK waves is
difficult due to unbounded perturbation $\beta_{x}\partial_{v}g_{\pm}$ in the
following eigenfunction equation at a BGK wave $\left(  f_{\pm}^{\beta}\left(
x,v\right)  ,\ \beta\right)  $ satisfying (\ref{defn-f-pn-BGK}) and
(\ref{bgk-uneven}):
\begin{align}
\lambda g_{\pm}+v\partial_{x}g_{\pm}\mp\beta_{x}\partial_{v}g_{\pm}\mp\phi
_{x}v\mu_{\pm,+}^{\prime} &  =0,\ v>0\label{vlasoveigen}\\
\lambda g_{\pm}+v\partial_{x}g_{\pm}\mp\beta_{x}\partial_{v}g_{\pm}\mp\phi
_{x}v\mu_{\pm,-}^{\prime} &  =0,\ v<0\nonumber
\end{align}%
\begin{equation}
-\phi_{xx}=\rho=\int\left(  g_{+}-g_{-}\right)  dv,\label{poissoneigen}%
\end{equation}
where $\lambda$ is the eigenvalue and $\left(  g_{\pm},\phi\right)  \ $is the
eigenfunction. For an unstable eigenvalue $\lambda$ with $\operatorname{Re}%
\lambda>0,$ it is possible to `integrate' the Vlasov equation
(\ref{vlasoveigen}) and study a nice operator on the electric potential $\phi$
to exclude unstable eigenvalues away from $0.$ The situation is much more
subtle when $\lambda$ is near zero, which is exactly the focus of the current
stability analysis.

There are three new ingredients in our resolution to such an open question.
The first is to use action-angle variables to integrate the Vlasov equation
(\ref{vlasoveigen}) and define a charge operator $\rho(\lambda;\varepsilon)$
acting on electric potential $\phi$, where $\varepsilon=\left\vert
\beta\right\vert _{^{\infty}}$. Remarkably, we observe that if $\mu_{\pm,\pm}$
are flat near the origin, then the operator $\rho(\lambda;\varepsilon)$ is
analytic in $\lambda$ near $0\ $and continuous in $\varepsilon.$ The second
ingredient is to prove the number of eigenvalue $\lambda$ near zero is no more
than two for the non-even BGK waves and no more than four for the even
BGK\ waves, thanks to a new abstract lemma on stability of eigenvalues and the
counting of the multiplicity of zero eigenvalue for the homogeneous state. The
last ingredient is to use the Hamiltonian structure of the linearized
Vlasov-Poisson system and the zero eigenmode due to translation. Combining
such a structure with the eigenvalue counting near $0$, we can rule out
unstable eigenvalues for the uneven steady states. For the even states, the
even and odd perturbations can be studied separately. For odd perturbations,
the unstable eigenvalues can be ruled out by the counting as in the uneven
case. For even perturbations, the possible unstable eigenvalues must be real,
from which a sharp stability criterion can be derived.

Since neutrally stable spectra can easily become unstable under perturbations,
it is generally difficult to construct stable steady states in Hamiltonian
systems via a perturbation method. Our successful construction provides a
general approach to find stability criteria to ensure that the zero eigenvalue
can only bifurcate to stable ones for Hamiltonian systems with certain natural
symmetry. The linearized Vlasov-Poisson system near BGK waves is a linear
Hamiltonian system $\mathcal{JL}$ (\ref{Hamiltonian}) with an indefinite
energy functional $\left\langle \mathcal{L}\vec{g},\vec{g}\right\rangle $, for
which there are very few methods to study the stability issues. Our approach
could be useful for other problems with an indefinite energy functional.

\section{Stability of Spectra}

First, we give an abstract lemma about the stability of eigenvalues near $0$.
Let $K(\lambda,\varepsilon)$ be a family of bounded linear operators from
Hilbert space $X$ to $X,$ where $\lambda\in\mathbf{C,}$ $\varepsilon
\in\mathbf{R}$. We assume that:%

\begin{equation}
\text{for }\varepsilon<<1,\ K(\lambda,\varepsilon)\ \text{is analytic
near\ }\lambda=0. \label{assumption A}%
\end{equation}
That is, for $\varepsilon<<1$, the map $\lambda\rightarrow K(\lambda
,\varepsilon)$ is analytic as an operator-valued function on a small disk
$B_{R\left(  \varepsilon\right)  }\left(  0\right)  \subset\mathbf{C}$. This
is equivalent to that $\lambda\rightarrow\left(  K(\lambda,\varepsilon
)u,v\right)  $ is analytic for any $u,v\in X$ (see \cite{hislop-sigal}). We
investigate the set of generalized eigenvalues $\Lambda^{\varepsilon
}=\{\lambda\in\mathbf{C}$\textbf{\ }such that there is $0\neq r\in X$ such
that $\left(  \mathbf{I}+K(\lambda,\varepsilon)\right)  r=0\}.$

\begin{lemma}
\label{lemma-stability-spectra}(Stability of Spectra) Assume
(\ref{assumption A}) and:

1) $\ker\{\mathbf{I}+K(0,0)\}=$span$\{r_{1},r_{2}\}.$

2) $\mathbf{I}+K(0,0):$ $($\textbf{$I$}$-\mathbf{P})X\rightarrow($\textbf{$I$%
}$-\mathbf{P})X\ $is invertible, where $\mathbf{P}$ is the projection to the
span of $\{r_{1},r_{2}\}$ and
\[
\det(\left(  \mathbf{I}+K(\lambda,0)\right)  r_{j},r_{i})\backsim\lambda
^{m},\ (i,j=1,2)
\]
$\ $near $\lambda=0.$

3) $K(\lambda,0)$ is continuous in $\lambda$ and $K(\lambda,0):($\textbf{$I$%
}$-\mathbf{P})X\rightarrow($\textbf{$I$}$-\mathbf{P})X$.

4) For any $\lambda\in\mathbf{C}$ with $\left\vert \lambda\right\vert <<1$,%
\[
\lim_{\varepsilon\rightarrow0}||K(\lambda,\varepsilon)-K(\lambda
,0)||_{L(X,X)}=0.
\]

Then there exists $\alpha>0$ such that for all $\varepsilon<<1,$ $\#\left(
\Lambda^{\varepsilon}\cap\left\{  \left\vert \lambda\right\vert <\alpha
\right\}  \right)  \leq m$.
\end{lemma}

\begin{proof}
The proof uses the Liapunov-Schmidt reduction on $\ker\{\mathbf{I}+K(0,0)\}$
and its complement space. By assumptions 2) and 3), there exists $\alpha>0$
such that for any $\lambda\neq0$ with $\left\vert \lambda\right\vert
<\alpha,\ \left(  \mathbf{I}+K(\lambda,0)\right)  |_{(\mathbf{I}-\mathbf{P}%
)X}$ is invertible. We assume $\left\vert \lambda\right\vert <\alpha$ below.

Let $\lambda\in\Lambda^{\varepsilon},$ so that there is $r_{\varepsilon}%
=r\neq0$ such that
\begin{equation}
\left(  \mathbf{I}+K(\lambda,\varepsilon)\right)  r=\{\mathbf{I}%
+K(\lambda,0)+[K(\lambda,\varepsilon)-K(\lambda,0)]\}r=0.
\label{eigenfunction-equation}%
\end{equation}
Let $r=r^{\perp}+r^{||}$, where $r^{||}=a_{1}r_{1}+a_{2}r_{2}$ is the
projection of $r$ to $\left\{  r_{1},r_{2}\right\}  $. Then projecting
(\ref{eigenfunction-equation}) to $($\textbf{$I$}$-\mathbf{P})X$, we get
\[
\left(  \mathbf{I}+K(\lambda,0)\right)  r^{\perp}+\left(  \mathbf{I}%
-\mathbf{P}\right)  [K(\lambda,\varepsilon)-K(\lambda,0)]\}\left(  r^{\perp
}+r^{||}\right)  =0.
\]
Solving $r^{\perp}$ in terms of $r^{||}$, we get
\begin{align*}
r^{\perp}  &  =-[\mathbf{I}+\left(  \mathbf{I}+K(\lambda,0)\right)
^{-1}\left(  \mathbf{I}-\mathbf{P}\right)  \left(  K(\lambda,\varepsilon
)-K(\lambda,0)\right)  ]^{-1}\\
&  \left(  \mathbf{I}+K(\lambda,0)\right)  ^{-1}\left(  \mathbf{I}%
-\mathbf{P}\right)  [K(\lambda,\varepsilon)-K(\lambda,0)]r^{||}\\
&  \equiv Z^{\perp}(\lambda,\varepsilon)\{a_{1}r_{1}+a_{2}r_{2}\},
\end{align*}
where $\mathbf{P}$ is the projection to $\ker\left(  \mathbf{I}+K(0,0)\right)
=\{r_{1},r_{2}\}$ and
\begin{align}
Z^{\perp}(\lambda,\varepsilon)  &  \equiv-[\mathbf{I}+\left(  \mathbf{I}%
+K(\lambda,0)\right)  ^{-1}\left(  \mathbf{I}-\mathbf{P}\right)  \left(
K(\lambda,\varepsilon)-K(\lambda,0)\right)  ]^{-1}\label{definition-Z-per}\\
&  \left(  \mathbf{I}+K(\lambda,0)\right)  ^{-1}\left(  \mathbf{I}%
-\mathbf{P}\right)  [K(\lambda,\varepsilon)-K(\lambda,0)].\nonumber
\end{align}
Plugging above formula to the equation (\ref{eigenfunction-equation}), we
have
\begin{align}
0  &  =[K(\lambda,\varepsilon)-K(\lambda,0)]\left(  Z^{\perp}(\lambda
,\varepsilon)+\mathbf{I}\right)  \left(  a_{1}r_{1}+a_{2}r_{2}\right)
\label{equation-LS-reduction}\\
&  \ \ +\left(  \mathbf{I}+K(\lambda,0)\right)  \left(  a_{1}r_{1}+a_{2}%
r_{2}\right)  +\left(  \mathbf{I}+K(\lambda,0)\right)  \left(  \left(
\mathbf{I}-\mathbf{P}\right)  r\right)  .\nonumber
\end{align}
Taking inner product of above equation with $r_{1}$ and $r_{2}$ respectively,
we get%
\begin{equation}
\sum_{j=1}^{2}(\left(  \mathbf{I}+K(\lambda,0)\right)  r_{j},r_{i})a_{j}%
+\sum_{j}^{2}B_{ij}(\lambda,\varepsilon)a_{j}=0,\ i=1,2,
\label{linear-system-eigenfunction}%
\end{equation}
where
\[
B_{ij}(\lambda,\varepsilon)=\left(  [K(\lambda,\varepsilon)-K(\lambda
,0)][Z^{\perp}(\lambda,\varepsilon)+\mathbf{I}]r_{j},r_{i}\right)  .
\]
Here, in the above we use the fact that
\[
((\mathbf{I}+K(\lambda,0))(\mathbf{I}-\mathbf{P})r,r_{i})=0,\ i=1,2.
\]
Define the $2$ by $2$ matrix $A(\lambda,\varepsilon)=\left(  A_{ij}%
(\lambda,\varepsilon)\right)  $ by
\[
A_{ij}=(\left(  \mathbf{I}+K(\lambda,0)\right)  r_{j},r_{i})+B_{ij}%
(\lambda,\varepsilon),~i,j=1,2,
\]
then the eigenvalue problem (\ref{eigenfunction-equation}) is equivalent to
$\det A(\lambda,\varepsilon)=0.$

By assumption (\ref{assumption A}), $K(\lambda,\varepsilon)$ is analytic near
$\lambda=0,$ it follows that $\det A(\lambda,\varepsilon)$ is analytic in
$\lambda$ near $0$ for $\varepsilon<<1.$ By 2),
\[
\det A(\lambda,0)=\det(\left(  \mathbf{I}+K(\lambda,0)\right)  r_{j}%
,r_{i})\backsim\lambda^{m}.
\]
Moreover, by 4) we have $\lim_{\varepsilon\rightarrow0}$ $|\det A(\lambda
,\varepsilon)-\det A(\lambda,0)|=0.$ It follows from the analytical function
theory, there exists $\alpha>0$ such that there are at most $m$ distinct
$\lambda$ with $\left\vert \lambda\right\vert <\alpha\ $satisfying $\det
A(\lambda,\varepsilon)=0$. Thus $\#\left(  \Lambda^{\varepsilon}\cap\left\{
\left\vert \lambda\right\vert <\alpha\right\}  \right)  \leq m$.
\end{proof}

By the same proof, we have a similar result when $\ker\left(  \mathbf{I}%
+K(0,0)\right)  $ is one-dimensional.

\begin{corollary}
\label{cor-stability-spectra}Assume (\ref{assumption A}) and:

1) $\ker\{\mathbf{I}+K(0,0)\}=$span$\{r\}.$

2) $\{\mathbf{I}+K(0,0)\}$ is invertible from $\{\mathbf{I}-\mathbf{P}%
\}X\rightarrow\{\mathbf{I}-\mathbf{P}\}X,$ where $\mathbf{P}$ is the
projection to span$\{r\},$ and $(\{I+K(\lambda,0)\}r,r)\backsim\lambda^{m}.$

3) $\{\mathbf{I}+K(\lambda,0)\}$ is continuous in $\lambda$ and $\{\mathbf{I}%
+K(\lambda,0)\}$ maps from $\{\mathbf{I}-\mathbf{P}\}X$ to $\{\mathbf{I}%
-\mathbf{P}\}X.$

4) \ For any $\lambda\in\mathbf{C}$, $\lim_{\varepsilon\rightarrow
0}||K(\lambda,\varepsilon)-K(\lambda,0)||\rightarrow0.$\ 

Then there exists a $\alpha>0$ such that if $|\varepsilon|<<1,$ $\#\{\Lambda
^{\varepsilon}\cap\{|\lambda|<\alpha\}\leq m.$
\end{corollary}

\section{Vlasov Spectra}

In this section, we use the Hamiltonian structure of the linearized
Vlasov-Poisson operator to show that the Vlasov spectra is symmetric to both
real and imaginary axes. We only consider the two-species case and the same is
true for the fixed ion case. Define the linearized Vlasov-Poisson operator
\begin{equation}
\mathcal{A}\left(
\begin{array}
[c]{c}%
g_{+}\\
g_{-}%
\end{array}
\right)  =\left(
\begin{array}
[c]{c}%
-\left(  v\partial_{x}-\beta_{x}\partial_{v}\right)  g_{+}+\partial_{x}%
\phi\partial_{v}\mu_{+}\\
-\left(  v\partial_{x}+\beta_{x}\partial_{v}\right)  g_{-}-\partial_{x}%
\phi\partial_{v}\mu_{-}%
\end{array}
\right)  , \label{defn-VP-operator}%
\end{equation}
with$\ $%
\[
\phi=\left(  \partial_{x}^{2}\right)  ^{-1}\left(  \int\{g_{+}-g_{-}%
\ \}dv\right)  ,
\]
\ \ where $\mu_{\pm}\equiv\mu_{\pm}(\frac{1}{2}|v|^{2}\pm\beta)$. Define
$X_{\pm}$ to be the $\left\vert \mu_{\pm}^{\prime}\right\vert \ $weighted
$L^{2}$ space and $X=X_{+}\times X_{-}$. We consider the spectra of the
operator $\mathcal{A}$ in $X$.

\begin{lemma}
\label{lemma-vlasov-spectra}(Structure of Vlasov Spectra) The essential
spectrum of the linearized Vlasov-Poisson system (\ref{vlasoveigen}) is the
imaginary axis. Let $\lambda$ be an eigenvalue, then both $\bar{\lambda
},-\lambda$ must also be eigenvalues.
\end{lemma}

\begin{proof}
We first note that the transport operator $-\left(  v\partial_{x}\mp\beta
_{x}\partial_{v}\right)  $ is an anti-symmetric closed operator with imaginary
axis being its essential spectrum, while $\pm\partial_{x}\phi\partial_{v}%
\mu_{\pm}$ with $\phi=\left(  \partial_{x}^{2}\right)  ^{-1}\left(
\int\{g_{+}-g_{-}\}dv\right)  $ is a relative compact perturbation. Thus by
Weyl's Theorem (Th. 5.35 in \cite{Kato}), the essential spectrum of
$\mathcal{A}$ remains the same, with possible additional discrete eigenvalues.

Define the operators
\[
\mathcal{L}_{\pm}g_{\pm}=\frac{g_{\pm}}{\mu_{\pm}^{\prime}},\ \mathcal{B}%
f=\left(  \partial_{x}^{2}\right)  ^{-1}\left(  \int fdv\right)
,\ \mathcal{J}_{\pm}=-\mu_{\pm}^{\prime}\left(  v\partial_{x}\mp\beta
_{x}\partial_{v}\right)  .
\]
Then formally the operator $\mathcal{A}=\mathcal{JL}\ $is of Hamiltonian form,
where the operators
\begin{equation}
\mathcal{J}=\left(
\begin{array}
[c]{cc}%
\mathcal{J}_{+} & 0\\
0 & \mathcal{J}_{-}%
\end{array}
\right)  ,\ \mathcal{L=}\left(
\begin{array}
[c]{cc}%
\mathcal{L}_{+}-\mathcal{B} & \mathcal{B}\\
\mathcal{B} & \mathcal{L}_{-}-\mathcal{B}%
\end{array}
\right)  \label{Hamiltonian}%
\end{equation}
are anti-symmetric and symmetric respectively on $X$. We choose $\lambda$ to
be an eigenvalue of $\mathcal{A}.$ Since $\mathcal{A}$ maps real functions to
real functions, if $\lambda$ is an eigenvalue of $\mathcal{A}$, then so is
$\bar{\lambda}$. To show that $-\lambda$ is also an eigenvalue of
$\mathcal{A}$, it suffices to assume that $\operatorname{Re}\lambda\neq0$
since otherwise $-\lambda=\bar{\lambda}$ is already an eigenvalue. Assume
$\operatorname{Re}\lambda>0$ and $\vec{g}=\left(  g_{+},g_{-}\right)  $ is an
eigenfunction. Then from $\mathcal{A}\vec{g}=\mathcal{J}\mathcal{L}\vec
{g}=\lambda\vec{g}$, clearly $\mathcal{L}\vec{g}\neq0$ since otherwise
$\lambda\vec{g}=\mathcal{J}\mathcal{L}\vec{g}=0$ so that $\lambda=0,$ a
contradiction. Since $\mathcal{J}^{\ast}=-\mathcal{J}$ and $\mathcal{L}^{\ast
}=\mathcal{L},$ $\mathcal{A}^{\ast}=-\mathcal{LJ}$. Define $\vec
{h}=\mathcal{L}\vec{g}\neq0$, then
\[
\mathcal{A}^{\ast}\vec{h}=-\left(  \mathcal{LJ}\right)  \mathcal{L}\vec
{g}=-\mathcal{L}\left(  \mathcal{JL}\right)  \vec{g}=-\lambda\mathcal{L}%
\vec{g}=-\lambda\vec{h},
\]
so $-\lambda$ is an eigenvalue of $\mathcal{A}^{\ast}.$ Since $\sigma
(\mathcal{A})=\overline{\sigma(\mathcal{A}^{\ast})},$ $-\bar{\lambda}$ and
therefore $-\lambda\ $is an eigenvalue of $\mathcal{A}.$
\end{proof}

\section{Action-Angle Reformulation}

To use Lemma \ref{lemma-stability-spectra} on the stability of spectra, we
reformulate the eigenvalue problem (\ref{vlasoveigen}) to a Fredholm operator
for the potential function $\phi$. To achieve this, we solve $f$ in terms of
$\phi$ by using the action-angle variables of the steady trajectory. Below, we
treat the two-species case. The fixed ion case is similar.

\textbf{Action-Angle Formulation: }The construction of the action-angle
variables follows from Section 50.B in \cite{arnold-mechanics}.

\textit{Inside separatrix}: When the initial state is in the trapped region,
that is,
\begin{equation}
\left(  x,v\right)  \in\Omega_{0}=\left\{  e_{\pm}<\max\beta=-\min
\beta\right\}  ,\label{defn-omego-0}%
\end{equation}
the particle is trapped in the interval $\left[  -\alpha_{\pm}(e_{\pm}%
),\alpha_{\pm}(e_{\pm})\right]  $, with the period
\[
T_{\pm}(e_{\pm})=2\int_{-\alpha_{\pm}(e_{\pm})}^{\alpha_{\pm}(e_{\pm})}%
\frac{dx^{\prime}}{\sqrt{2(e_{\pm}\mp\beta(x^{\prime}))}},
\]
where $e_{\pm}=\pm\beta(\alpha_{\pm}(e_{\pm}))$. Define the action variable
\[
I_{\pm}(e_{\pm})=\frac{1}{2\pi}\int_{-\max\beta}^{e_{\pm}}T_{\pm}(e_{\pm
}^{\prime})de_{\pm}^{\prime},
\]
and the angle variable
\[
\theta_{\pm}=\frac{2\pi}{T_{\pm}(e_{\pm})}\int_{-\alpha_{\pm}}^{x}%
\frac{dx^{\prime}}{\sqrt{2(e_{\pm}\mp\beta(x^{\prime}))}},\ \ v>0,
\]
and
\[
\theta_{\pm}=2\pi-\frac{2\pi}{T_{\pm}(e_{\pm})}\int_{-\alpha_{\pm}}^{x}%
\frac{dx^{\prime}}{\sqrt{2(e_{\pm}\mp\beta(x^{\prime}))}},\ \ v<0.
\]

\textit{On the separatrix}:

When $\left(  x,v\right)  \in\left\{  e_{\pm}=-\min\beta\right\}  $, the
particle takes infinite time to approach the saddle point $\left(
\frac{P_{\beta}}{2},0\right)  $ for electrons and $\left(  0,0\right)  $ for ions.

\textit{Outside separatrix}:

When the initial state is in the upper untrapped region, that is,
\begin{equation}
\left(  x,v\right)  \in\Omega_{+}=\left\{  e_{\pm}>-\min\beta,v>0\right\}
,\label{defn-omega-+}%
\end{equation}
or in the lower untrapped region, that is,
\begin{equation}
\left(  x,v\right)  \in\Omega_{-}=\left\{  e_{\pm}>-\min\beta,v<0\right\}
,\label{defn-omega-}%
\end{equation}
the particle goes through the whole interval $\left[  0,P_{\beta}\right]  $
without changing its direction. Then the period of the particle motion is
\[
T_{\pm}(e_{\pm})=\int_{0}^{P_{\beta}}\frac{dx^{\prime}}{\sqrt{2(e_{\pm}%
\mp\beta(x^{\prime}))}}.
\]
We define the action and angle variables by
\begin{align}
\text{ }I_{\pm}(e_{\pm}) &  =\frac{1}{2\pi}\int_{-\min\beta}^{e_{\pm}}T_{\pm
}(e_{\pm}^{\prime})de_{\pm}^{\prime},\label{actionangle}\\
\theta_{\pm} &  =\frac{2\pi}{T_{\pm}(e_{\pm})}\int_{0}^{x}\frac{dx^{\prime}%
}{\sqrt{2(e_{\pm}\mp\beta(x^{\prime}))}}\text{ },\nonumber
\end{align}
and denote
\[
\omega_{\pm}(I_{\pm})=\frac{2\pi}{T_{\pm}(e_{\pm}\left(  I_{\pm}\right)  )}%
\]
to be the frequency. We list some basic properties of action-angle variables
(see \cite{arnold-mechanics}). First, for both trapped region $\Omega_{0}$ and
untrapped regions $\Omega_{\pm}$, the action-angle transform $(x,v)\rightarrow
(I_{\pm},\theta_{\pm})$ is a smooth diffeomorphism with Jacobian $1$. Second,
in the coordinates $\left(  I_{\pm},\theta_{\pm}\right)  $, the particle
motion equation $\dot{X}_{\pm}=V_{\pm},$ $\dot{V}_{\pm}=\mp\beta_{x}\left(
X_{\pm}\right)  $ becomes $\dot{I}_{\pm}=0,\ \dot{\theta}_{\pm}=\omega_{\pm
}(I_{\pm})$ for trapped particles; for free particle, it becomes $\dot{I}%
_{\pm}=0,\ \dot{\theta}_{\pm}=\omega_{\pm}(I_{\pm})$, when $V_{\pm}\left(
0\right)  >0$ and $\dot{I}_{\pm}=0,\ \dot{\theta}_{\pm}=-\omega_{\pm}(I_{\pm
})$, when $V_{\pm}\left(  0\right)  <0$. So the particle trajectory $\left(
X_{\pm}\left(  t;x,v\right)  ,V_{\pm}\left(  t;x,v\right)  \right)  $ becomes:
$\left(  I_{\pm},\ \theta_{\pm}+t\omega_{\pm}(I_{\pm})\right)  $ inside the
separatrix; $\left(  I_{\pm},\ \theta_{\pm}+t\omega_{\pm}(I_{\pm})\right)  $
for $v>0$ and $\left(  I_{\pm},\ \theta_{\pm}-t\omega_{\pm}(I_{\pm})\right)  $
for $v<0$ , outside the separatrix. Here, $\left(  I_{\pm},\ \theta_{\pm
}\right)  $ are the action-angle variables for the initial position $\left(
X_{\pm}\left(  0\right)  ,V_{\pm}\left(  0\right)  \right)  =\left(
x,v\right)  $. Correspondingly, we have the following relations of the
transport operators in $\left(  x,v\right)  $ and $\left(  I_{\pm}%
,\ \theta_{\pm}\right)  $:
\begin{equation}
v\partial_{x}\mp\beta_{x}\partial_{v}=\omega_{\pm}(I_{\pm})\partial
_{\theta_{\pm}}\label{transport-inside}%
\end{equation}
inside the separatrix, and
\begin{equation}
v\partial_{x}\mp\beta_{x}\partial_{v}=\left\{
\begin{array}
[c]{cc}%
\omega_{\pm}(I_{\pm})\partial_{\theta_{\pm}} & \text{for }v>0\\
-\omega_{\pm}(I_{\pm})\partial_{\theta_{\pm}} & \text{for }v<0
\end{array}
\right.  \text{ }\label{transport-outside}%
\end{equation}
outside the separatrix. We summarize main properties of\ the action-angle
transform in the following lemma.

\begin{lemma}
\label{lemma-action-angle}In the angle-action variables $(I_{\pm},\theta_{\pm
})$, we have

(i) $0\leq\omega_{\pm}<\infty,$

(ii) $\lim_{e_{\pm}\rightarrow-\min\beta}T_{\pm}(e_{\pm})=\infty,$
$\lim_{e_{\pm\rightarrow-\min\beta}}\omega_{\pm}(e_{\pm})=0,$

(iii) Inside the trapped region $\Omega_{0}$:%
\[
v\partial_{x}\mp\beta_{x}\partial_{v}=\omega_{\pm}(I_{\pm})\partial
_{\theta_{\pm}}.
\]

(iv) Outside the trapped region:
\begin{align*}
v\partial_{x}\mp\beta_{x}\partial_{v}  &  =\omega_{\pm}(I_{\pm})\partial
_{\theta_{\pm}},\text{ \ \ for }v>0,\\
\text{ }v\partial_{x}\mp\beta_{x}\partial_{v}  &  =-\omega_{\pm}(I_{\pm
})\partial_{\theta_{\pm}},\text{ \ for }v<0.
\end{align*}

\end{lemma}

Recall that in this paper, the profiles $\mu_{\pm}\ $of BGK waves are assumed
to be flat near zero in an interval $\left[  -\sigma_{\pm},\sigma_{\pm
}\right]  $. Let $\sigma=\min\left\{  \sigma_{+},\sigma_{-}\right\}  $. Below,
the notation $f\lesssim g$ $\left(  f\gtrsim g\right)  \ $stands for $f\leq
Cg$ $\left(  f\geq Cg\right)  $, for a generic constant $C>0$ independent of
$\varepsilon=\left\vert \beta\right\vert _{\infty}$.

\begin{lemma}
\label{epestimate}Assume $\varepsilon=\left\vert \beta\right\vert _{\infty
}<<\sigma$. For $|e_{\pm}|\geq\frac{\sigma}{2}$ (outside of separatrix)$,$ we
have $\left\vert \frac{\partial x}{\partial\theta}\right\vert \lesssim1,$
\[
\left\vert \omega_{\pm}^{\prime}(I_{\pm})-\left(  \frac{2\pi}{P_{\beta}%
}\right)  ^{2}\right\vert ,\ |\partial_{I_{\pm}}x|,\ \left\vert \omega_{\pm
}(I_{\pm})-\frac{2\pi}{P_{0}}\left\vert v\right\vert \right\vert ,\ \left\vert
\frac{d}{dI_{\pm}}\left(  \frac{1}{\omega_{\pm}^{\prime}(I_{\pm})}\right)
\right\vert ,\ \left\vert x-\frac{P_{\beta}}{2\pi}\theta_{\pm}\right\vert
\lesssim\varepsilon,
\]
and
\begin{equation}
\omega_{\pm}(I_{\pm})=\frac{2\pi}{P_{\beta}}\sqrt{2e_{\pm}}+O(\frac
{\varepsilon}{\sqrt{2e_{\pm}}}). \label{expansion-w}%
\end{equation}

\end{lemma}

\begin{proof}
Fix $e_{\pm}\geq\frac{\sigma}{2}>>\varepsilon,$ from $\beta=\varepsilon
\cos\frac{2\pi}{P_{\beta}}x+O(\varepsilon^{2}),$ we get \textit{\ }%
\begin{equation}
T_{\pm}(e_{\pm})=\int_{0}^{P_{\beta}}\frac{dx^{\prime}}{\sqrt{2(e_{\pm}%
\mp\beta(x^{\prime}))}}=\frac{P_{\beta}}{\sqrt{2e_{\pm}}}+O(\frac{\varepsilon
}{\left(  e_{\pm}\right)  ^{\frac{3}{2}}}),\label{texpansion}%
\end{equation}
since by Taylor expansion%
\begin{equation}
\breve{\{}2(e_{\pm}\mp\beta(x^{\prime}))\}^{-1/2}=\{2e_{\pm}\}^{-1/2}%
\{1\mp\frac{\varepsilon}{2e_{\pm}}\cos\frac{2\pi x^{\prime}}{P_{\beta}%
}+O\left(  \left(  \frac{\varepsilon}{e_{\pm}}\right)  ^{2}\right)
.\label{expansion-sqrt}%
\end{equation}
Similarly, since $\frac{P_{\beta}}{T_{\pm}(e_{\pm})\sqrt{2e_{\pm}}}%
=1+O(\frac{\varepsilon}{e_{\pm}})$,
\begin{align}
\theta_{\pm} &  =\frac{2\pi}{T_{\pm}(e_{\pm})}\int_{0}^{x}\frac{dx^{\prime}%
}{\sqrt{2(e_{\pm}\mp\beta(x^{\prime}))}}\label{thetaexpansion}\\
&  =\frac{2\pi}{T_{\pm}(e)\sqrt{2e_{\pm}}}\left\{  x\mp\frac{\varepsilon
P_{\beta}}{4\pi e_{\pm}}\sin\frac{2\pi}{P_{\beta}}x+O(\left(  \frac
{\varepsilon}{e_{\pm}}\right)  ^{2})\right\}  \nonumber\\
&  =\frac{2\pi}{P_{\beta}}x\mp\frac{\varepsilon}{2e_{\pm}}\sin\frac{2\pi
}{P_{\beta}}x+O(\varepsilon^{2}),\nonumber
\end{align}
and thus$\ \left\vert \frac{\partial x}{\partial\theta}\right\vert
\lesssim1,\left\vert x-\frac{P_{\beta}}{2\pi}\theta_{\pm}\right\vert
,\partial_{I_{\pm}}x\lesssim\varepsilon$. By (\ref{texpansion}) and
(\ref{expansion-sqrt}), we have
\[
\omega_{\pm}(I_{\pm})=\frac{2\pi}{P_{\beta}}\sqrt{2e_{\pm}}+O(\frac
{\varepsilon}{\sqrt{2e_{\pm}}}).
\]
Combined above with $\left\vert v\right\vert =\sqrt{2e_{\pm}}+O(\varepsilon)$,
we get $\left\vert \omega_{\pm}(I_{\pm})-\frac{2\pi}{P_{0}}\left\vert
v\right\vert \right\vert \lesssim\varepsilon$. Note that
\[
\frac{dI_{\pm}}{de_{\pm}}=\frac{T_{\pm}(e_{\pm})}{2\pi}=\frac{1}{\omega_{\pm
}(I_{\pm})},
\]
and from (\ref{texpansion})%
\[
T_{\pm}^{\prime}(e_{\pm})=-\frac{P_{\beta}}{\left(  \sqrt{2e_{\pm}}\right)
^{3}}+O(\varepsilon)=-\frac{1}{P_{\beta}^{2}}T_{\pm}^{3}(e_{\pm}%
)+O(\varepsilon),
\]
so
\[
\omega_{\pm}^{\prime}(I_{\pm})=-\frac{2\pi}{T_{\pm}^{2}\left(  e_{\pm}\right)
}T_{\pm}^{\prime}(e_{\pm})\frac{de_{\pm}}{dI_{\pm}}=\left(  \frac{2\pi
}{P_{\beta}}\right)  ^{2}+O(\varepsilon)
\]
which implies that
\[
\left\vert \omega_{\pm}^{\prime}(I_{\pm})-\left(  \frac{2\pi}{P_{\beta}%
}\right)  ^{2}\right\vert ,\ \left\vert \frac{d}{dI_{\pm}}\left(  \frac
{1}{\omega_{\pm}^{\prime}(I_{\pm})}\right)  \right\vert \lesssim\varepsilon.
\]
This finishes the proof of the lemma.
\end{proof}

\textbf{The density operator }$\rho(\lambda,\varepsilon):\ $Consider a BGK
wave solution
\[
\lbrack\mu_{+,\pm}(e_{+}),\mu_{-,\pm}(e_{-}),-\beta_{x}]
\]
as in Theorems \ref{thm-main-uneven} and \ref{thm-main-even}. For an
eigenvalue $\lambda=a+bi,$ let $\left(  g_{\pm},\phi\right)  $ be the
eigenfunction satisfying (\ref{vlasoveigen}) and (\ref{poissoneigen}). We will
do the Fourier expansion of $g_{\pm}$ and $\phi\ $in both action-angle
variables $(I_{\pm},\theta_{\pm})$.

Define the spaces
\[
H_{\varepsilon}^{1}=\{P_{\beta}-\text{periodic\ }H^{1}\ \text{functions with
zero mean}\},
\]
and
\[
H_{0}^{1}=\{P_{0}-\text{periodic\ }H^{1}\ \text{functions with zero mean}\}.
\]
For any potential $\phi\in H_{\varepsilon}^{1},$ we expand
\begin{equation}
\phi=\sum_{k\in\mathbf{Z}}\phi_{k}^{\pm}(I_{\pm})e^{ik\theta_{\pm}},\text{
\ \ where\ }\phi_{k}^{\pm}(I_{\pm})=\frac{1}{2\pi}\int_{0}^{2\pi}%
\phi(x)e^{-ik\theta_{\pm}}d\theta_{\pm}. \label{expansion-phi}%
\end{equation}

Then, we expand
\begin{equation}
g_{\pm}=\sum g_{k}^{\pm}(I_{\pm})e^{ik\theta_{\pm}},\ k\in\mathbf{Z}.
\label{expansion-g-pn}%
\end{equation}

\textit{Inside the separatrix\ }($e_{\pm}(x,v)<-\min\beta$): Since $|e_{\pm
}(x,v)|=\left\vert \frac{1}{2}v^{2}\pm\beta\right\vert \leq\varepsilon
<<\sigma$ so that $\phi_{x}v\mu_{\pm,\pm}^{\prime}\equiv0$ in
(\ref{vlasoveigen}), thanks to the flatness assumption for $\mu_{\pm,\pm}$. By
(\ref{transport-inside}),
\begin{equation}
v\partial_{x}g_{\pm}\mp\beta_{x}\partial_{v}g_{\pm}=\omega_{\pm}(I_{\pm
})\partial_{\theta_{\pm}}g_{\pm},\label{transport-action-angle}%
\end{equation}
in the trapped region, so $\lambda g_{\pm}+\omega_{\pm}(I_{\pm})\partial
_{\theta_{\pm}}g_{\pm}=0$ from (\ref{vlasoveigen}). We therefore deduce that
$g_{k}^{\pm}(I_{\pm})\equiv0\ $inside separatrix, for all $k\in\mathbf{Z.}$

\textit{Outside the separatrix (}$e_{\pm}(x,v)>-\min\beta$):$\ $By
(\ref{transport-outside}), the equation (\ref{vlasoveigen}) becomes
\begin{align}
\lambda g_{+}\pm\omega_{+}(I_{+})\partial_{\theta_{+}}g_{+}\mp\mu_{+,\pm
}^{\prime}(e_{+})\omega_{+}(I_{+})\partial_{\theta_{+}}\phi &  =0,\text{ for
}\pm v>0,\label{vlasovangle}\\
\lambda g_{-}\pm\omega_{-}(I_{-})\partial_{\theta_{-}}g_{-}\pm\mu_{-,\pm
}^{\prime}(e_{-})\omega_{-}(I_{-})\partial_{\theta_{-}}\phi &  =0,\text{ for
}\pm v>0, \label{g-}%
\end{align}
in the untrapped region. We may solve (\ref{vlasovangle})-(\ref{g-}) by using
the expansions (\ref{expansion-phi})-(\ref{expansion-g-pn}) to obtain: If
$\lambda\neq0$, then $g_{k}^{\pm}(I)\equiv0$ for $k=0$ and for $k\neq0$,
\begin{equation}
g_{k}^{+}(I_{+})=\frac{\mu_{+,\pm}^{\prime}(e_{+})\omega_{+}\left(
I_{+}\right)  \phi_{k}^{+}(I_{+})}{\omega_{+}\pm\lambda/ik}\text{\ \ for }\pm
v>0, \label{vlasovsolver+}%
\end{equation}
and
\begin{equation}
g_{k}^{-}(I_{-})=-\frac{\mu_{-,\pm}^{\prime}(e_{-})\omega_{-}\left(
I_{-}\right)  \phi_{k}^{-}(I_{-})}{\omega_{+}\pm\lambda/ik}\text{\ \ for }\pm
v>0. \label{vlasovsolver-}%
\end{equation}
\ \ We note that (\ref{vlasovsolver+})-(\ref{vlasovsolver-}) are also valid
inside the separatrix thanks to the flatness of $\mu_{\pm,\pm}$.

By using (\ref{vlasovsolver+})-(\ref{vlasovsolver-})$,$ we define the charge
density operator as
\begin{align}
&  \ \ \rho(\lambda,\varepsilon)\phi\label{k}\\
&  =\sum_{k\neq0,\ \pm}\left(  \int_{v>0}e^{ik\theta_{\pm}}\frac{\omega_{\pm
}\mu_{\pm,+}^{\prime}(e_{\pm})}{\omega_{\pm}+\frac{\lambda}{ik}}\phi_{k}^{\pm
}(I_{\pm})dv+\int_{v<0}e^{ik\theta_{\pm}}\frac{\omega_{\pm}\mu_{\pm,-}%
^{\prime}(e_{\pm})}{\omega_{\pm}-\frac{\lambda}{ik}}\phi_{k}^{\pm}(I_{\pm
})dv\right)  .\nonumber
\end{align}
In the formula (\ref{k}), we note that $\left(  I_{\pm},\theta_{\pm}\right)  $
in the right hand side can be restricted to the untrapped region since
$\mu_{\pm,-}^{\prime}(e_{\pm}),\ \mu_{\pm,+}^{\prime}(e_{\pm})=0$ in the
trapped region. We also note that for any $\lambda\neq0$,
\begin{align*}
\int_{0}^{P_{\beta}}\rho(\lambda,\varepsilon)\phi\ dx  &  =\int\int\left(
g_{+}-g_{-}\right)  \ dxdv\\
&  =\int\int g_{+}dI_{+}d\theta_{+}-\int\int g_{-}dI_{-}d\theta_{-}=0,
\end{align*}
by using that $g_{k}^{\pm}=0$ when $k=0$. Hence $-\partial_{x}^{-2}%
\rho(\lambda,\varepsilon)\phi$ is well-defined and the self-consistent Poisson
equation (\ref{poissoneigen}) is reduced to $(\mathbf{I}+\partial_{x}^{-2}%
\rho(\lambda,\varepsilon))\phi=0.$ So we conclude that for nonzero
eigenvalues, the eigenspaces of (\ref{vlasoveigen})-(\ref{poissoneigen}) are
equivalent to null spaces of the operator
\[
\mathbf{I}+\partial_{x}^{-2}\rho(\lambda,\varepsilon):H_{\varepsilon}%
^{1}\rightarrow H_{\varepsilon}^{1},
\]
where $\partial_{x}^{-2}$ denotes twice anti-derivatives with zero mean. To
apply Lemma \ref{lemma-stability-spectra}, we rescale above operators to be
defined in the same function space $H_{0}^{1}$. Let $G_{\beta}^{-1}\phi
=\phi(\frac{P_{0}}{P_{\beta}}x)$ be the mapping from $H_{0}^{1}\rightarrow
H_{\varepsilon}^{1},$ and define the operator%
\[
K(\lambda,\varepsilon)=G_{\beta}\partial_{x}^{-2}\rho(\lambda,\varepsilon
)G_{\beta}^{-1}%
\]
in $H_{0}^{1}$.

To study the properties of the operators $K(\lambda,\varepsilon)$, we
introduce two lemmas.

\begin{lemma}
\label{lemma-hardy} If $u\left(  v\right)  \in W^{s,p}\left(  \mathbf{R}%
\right)  $ $\left(  p>1,s>\frac{1}{p}\right)  ,\ $then for any $z\in
\mathbf{C}$ with $\operatorname{Re}z\neq0$, we have
\[
\left\vert \int_{\mathbf{R}}\frac{u\left(  v\right)  }{v-z}dv\right\vert \leq
C\left\Vert u\right\Vert _{W^{s,p}\left(  \mathbf{R}\right)  },
\]
for some constant $C$ independent of $z$.
\end{lemma}

\begin{proof}
Let $z=a+ib$, with $a,b\in\mathbf{R},\ b\neq0$. Then
\[
\int_{\mathbf{R}}\frac{u\left(  v\right)  }{v-z}dv=\int_{\mathbf{R}}%
\frac{u\left(  v\right)  \left(  v-a\right)  +ibu\left(  v\right)  }{\left(
v-a\right)  ^{2}+b^{2}}dv.
\]
Since$\ s>\frac{1}{p}$, the space $W^{s,p}\left(  \mathbf{R}\right)  $ is
embedded to the H\"{o}lder space $C^{0,\gamma}$ with $\gamma\in\left(
0,s-\frac{1}{p}\right)  $. So
\[
\left\vert u\left(  v\right)  -u\left(  a\right)  \right\vert \leq\left\vert
v-a\right\vert ^{\gamma}\left\Vert u\right\Vert _{C^{0,\alpha}}\leq
C\left\Vert u\right\Vert _{W^{s,p}}\left\vert v\right\vert ^{\gamma}.
\]
The real part is estimated by
\begin{align*}
&  \left\vert \int_{\mathbf{R}}\frac{u\left(  v\right)  \left(  v-a\right)
}{\left(  v-a\right)  ^{2}+b^{2}}dv\right\vert \\
&  \leq\left\vert \int_{-1+a}^{1+a}\frac{\left(  u\left(  v\right)  -u\left(
a\right)  \right)  \left(  v-a\right)  }{\left(  v-a\right)  ^{2}+b^{2}%
}dv\right\vert +\int_{\left\vert v-a\right\vert \geq1}\left\vert
\frac{u\left(  v\right)  }{v-a}\right\vert dv\\
&  \leq\int_{-1+a}^{1+a}\left\vert \frac{u\left(  v\right)  -u\left(
a\right)  }{v-a}\right\vert dv+\left(  \int_{\left\vert v-a\right\vert \geq
1}\frac{1}{\left\vert v-a\right\vert ^{p^{\prime}}}dv\right)  ^{\frac
{1}{p^{\prime}}}\left\Vert u\right\Vert _{L^{p}}\\
&  \lesssim\left\Vert u\right\Vert _{W^{s,p}}\int_{-1+a}^{1+a}\left\vert
v-a\right\vert ^{-1+\gamma}dv+\left\Vert u\right\Vert _{L^{p}}\lesssim
\left\Vert u\right\Vert _{W^{s,p}}.
\end{align*}
Similarly, for the imaginary part, we have
\begin{align*}
&  \left\vert \int_{\mathbf{R}}\frac{bu\left(  v\right)  }{\left(  v-a\right)
^{2}+b^{2}}dv\right\vert \\
&  \leq\left\vert \int_{-1+a}^{1+a}\frac{\left(  u\left(  v\right)  -u\left(
a\right)  \right)  b}{\left(  v-a\right)  ^{2}+b^{2}}dv\right\vert
+\pi\left\vert u\left(  a\right)  \right\vert +\int_{\left\vert v-a\right\vert
\geq1}\left\vert \frac{u\left(  v\right)  }{v-a}\right\vert dv\\
&  \lesssim\left\Vert u\right\Vert _{W^{s,p}}.
\end{align*}
Here, in the above we use
\[
\left\vert \int_{-1+a}^{1+a}\frac{b}{\left(  v-a\right)  ^{2}+b^{2}%
}dv\right\vert \leq\int_{\mathbf{R}}\frac{1}{1+y^{2}}dy=\pi.
\]

\end{proof}

\begin{lemma}
\label{lemma-bound-phi-k}Given $\phi\in H^{1}\left(  0,P_{\beta}\right)  $ and
$\phi=\sum_{k}\phi_{k}^{\pm}(I_{\pm})e^{ik\theta_{\pm}}$,

i) If $w:\mathbf{R}^{+}\rightarrow\mathbf{R}^{+}$ and $\int w\left(  \frac
{1}{2}v^{2}\right)  dv<\infty$, then
\begin{equation}
\sum_{k}\int w\left(  e_{\pm}\right)  \left\vert \phi_{k}^{\pm}(I_{\pm
})\right\vert ^{2}dI_{\pm}\lesssim\left\Vert \phi\right\Vert _{L^{2}}%
^{2},\label{inequality-phi-k}%
\end{equation}

and%
\begin{equation}
\sum_{k}\int w\left(  e_{\pm}\right)  \left\vert \phi_{k}^{\pm\prime}(I_{\pm
})\right\vert ^{2}dI_{\pm}\lesssim\left\Vert \beta\right\Vert _{L^{\infty}%
}^{2}\left\Vert \phi_{x}\right\Vert _{L^{2}}^{2}%
.\label{inequality-phi-k-prime}%
\end{equation}

2) If $w:\mathbf{R}^{+}\rightarrow\mathbf{R}^{+}$ and $\int v^{2}w\left(
\frac{1}{2}v^{2}\right)  dv<\infty$, then%
\[
\sum_{k\neq0}k^{2}\int\omega_{\pm}^{2}\left(  I_{\pm}\right)  w\left(  e_{\pm
}\right)  \left\vert \phi_{k}^{\pm}(I_{\pm})\right\vert ^{2}dI_{\pm}%
\lesssim\left\Vert \phi_{x}\right\Vert _{L^{2}}^{2}.
\]

\end{lemma}

\begin{proof}
Proof of i): Since $\left(  x,v\right)  \rightarrow\left(  I_{\pm},\theta
_{\pm}\right)  $ has Jacobian $1$, so
\begin{align*}
\sum_{k}\int w\left(  e_{\pm}\right)  \left\vert \phi_{k}^{\pm}(I_{\pm
})\right\vert ^{2}dI_{\pm} &  =\int\int w\left(  e_{\pm}\right)  \left\vert
\phi\left(  x\right)  \right\vert ^{2}dxdv\\
&  \leq\sup_{x}\int w\left(  e_{\pm}\right)  dv\left\Vert \phi\right\Vert
_{L^{2}}^{2}\lesssim\left\Vert \phi\right\Vert _{L^{2}}^{2}.
\end{align*}
To prove (\ref{inequality-phi-k-prime}), we notice that
\[
\phi_{k}^{\pm\prime}(I_{\pm})=\frac{1}{2\pi}\int_{0}^{2\pi}e^{-ik\theta_{\pm}%
}\phi^{\prime}(x)\frac{\partial x}{\partial I_{\pm}}d\theta_{\pm},
\]
and by Lemma \ref{epestimate}, $\left\vert \frac{\partial x}{\partial I_{\pm}%
}\right\vert \lesssim\varepsilon$. So
\begin{align*}
\sum_{k}\int w\left(  e_{\pm}\right)  \left\vert \phi_{k}^{\pm\prime}(I_{\pm
})\right\vert ^{2}dI_{\pm} &  =\int\int w\left(  e_{\pm}\right)  \left\vert
\phi^{\prime}(x)\right\vert ^{2}\left\vert \frac{\partial x}{\partial I_{\pm}%
}\right\vert ^{2}dxdv\\
&  \lesssim\varepsilon^{2}\sup_{x}\int w\left(  e_{\pm}\right)  dv\left\Vert
\phi^{\prime}\right\Vert _{L^{2}}^{2}\lesssim\left\Vert \beta\right\Vert
_{L^{\infty}}^{2}\left\Vert \phi^{\prime}\right\Vert _{L^{2}}^{2}.
\end{align*}

Proof of ii): We note that by Lemma \ref{lemma-action-angle},
\[
v\phi_{x}=v\partial_{x}\mp\beta_{x}\partial_{v}=\left\{
\begin{array}
[c]{cc}%
\omega_{\pm}(I_{\pm})\partial_{\theta_{\pm}}\phi & \text{when }\left(
x,v\right)  \in\Omega_{0}\cup\Omega_{+}\\
-\omega_{\pm}(I_{\pm})\partial_{\theta_{\pm}}\phi & \text{when }\left(
x,v\right)  \in\Omega_{-}%
\end{array}
\right.  ,
\]
where $\Omega_{0},\Omega_{\pm}$ are defined in (\ref{defn-omego-0}),
(\ref{defn-omega-+}), (\ref{defn-omega-}). Thus,
\begin{align*}
&  \ \ \ \ \ \sum_{k\neq0}k^{2}\int\omega_{\pm}^{2}\left(  I_{\pm}\right)
w\left(  e_{\pm}\right)  \left\vert \phi_{k}^{\pm}(I_{\pm})\right\vert
^{2}dI_{\pm}\\
&  =\int\int w\left(  e_{\pm}\right)  \left\vert v\phi_{x}\right\vert
^{2}\ dxdv\\
&  \leq\sup_{x}\int\int w\left(  e_{\pm}\right)  \left\vert v\right\vert
^{2}\ dv\ \left\Vert \phi^{\prime}\right\Vert _{L^{2}}^{2}\lesssim\left\Vert
\phi_{x}\right\Vert _{L^{2}}^{2}.
\end{align*}

\end{proof}

\begin{lemma}
\label{Lemma-operator-analytic}

(i) $0\neq\lambda$ is an eigenvalue of the linearized VP operator iff there
exists $0\neq\phi\in H_{0}^{1}\ $such that $\{I+K(\lambda,\varepsilon
)\}\phi=0$.

(ii) For $\varepsilon<<\sigma$ and any $\lambda\in\mathbf{C}$ with
$\operatorname{Re}\lambda>0$, the operator $K(\lambda,\varepsilon):$
$H_{0}^{1}\rightarrow H_{0}^{1}\ $is uniformly bounded to $\lambda$.

(iii) For $\varepsilon<<\sigma$, $K(\lambda,\varepsilon):H_{0}^{1}\rightarrow
H_{0}^{1}$ is analytic in $\lambda$ when $|\lambda|<<1$.
\end{lemma}

\begin{proof}
(i) follows from the definition of the operator $K\left(  \lambda
,\varepsilon\right)  $. Indeed, $\lambda$ is an non-zero eigenvalue for the
Vlasov-Poisson operator satisfying (\ref{vlasoveigen}) (\ref{poissoneigen}),
if and only if (\ref{vlasovangle}), (\ref{g-}), (\ref{vlasovsolver+}) and
(\ref{vlasovsolver-}) are valid, which by the Poisson equation
(\ref{poissoneigen}) is equivalent to $\phi_{xx}=\rho(\lambda,\varepsilon
)\phi$ or $(I+\partial_{x}^{-2}\rho(\lambda,\varepsilon))\phi=0.$

Proof of (ii): By the change of variable from $\left[  0,P_{0}\right]  $ to
$\left[  0,P_{\beta}\right]  $, it is equivalent to show that the operator
\[
\partial_{x}^{-2}\rho(\lambda,\varepsilon):H_{\varepsilon}^{1}\rightarrow
H_{\varepsilon}^{1}%
\]
is uniformly bounded. For any $\phi\in H_{\varepsilon}^{1}$, let
$\Phi=\partial_{x}^{-2}\rho(\lambda,\varepsilon)\phi$. Since $\Phi\left(
x\right)  $ has zero mean, so
\[
\left\Vert \Phi\right\Vert _{H^{1}}\approx\left\Vert \Phi_{x}\right\Vert
_{L^{2}}=\sup_{\substack{\psi\in L^{2}\\\left\Vert \psi\right\Vert _{L^{2}}%
=1}}\left(  \Phi_{x},\psi\right)  .
\]
Let $\Psi=\partial_{x}^{-1}\psi$ be the anti-derivative of $\psi$ with zero
mean, then $\left\Vert \Psi\right\Vert _{H^{1}}\approx\left\Vert
\psi\right\Vert _{L^{2}}$. Let$\ $%
\[
\Psi(x)=\sum_{k}e^{ik\theta_{\pm}}\Psi_{k}^{\pm}(I_{\pm}),
\]
then
\begin{align}
&  \left(  \Phi_{x},\psi\right)  =-\left(  \Phi_{xx},\Psi\right)  =-\left(
\rho(\lambda,\varepsilon)\phi,\Psi\right)  \label{integral-L2-norm}\\
&  =-\sum_{k\neq0,\ \pm}\left(  \int\frac{\omega_{\pm}\mu_{\pm,+}^{\prime}%
}{\omega_{\pm}+\frac{\lambda}{ik}}\phi_{k}^{\pm}(I_{\pm})\bar{\Psi}_{k}^{\pm
}(I_{\pm})dI_{\pm}+\int\frac{\omega_{\pm}\mu_{\pm,-}^{\prime}}{\omega_{\pm
}-\frac{\lambda}{ik}}\phi_{k}^{\pm}\bar{\Psi}_{k}^{\pm}dI_{\pm}\right)
.\nonumber
\end{align}
To estimate the above integrals, we change the integration variable to
$\omega_{\pm}$ with $dI_{\pm}=\frac{1}{\omega_{\pm}^{\prime}\left(  I_{\pm
}\right)  }d\omega_{\pm}$. Define
\begin{equation}
H_{k}^{\pm,+}\left(  \omega_{\pm}\right)  =\omega_{\pm}\mu_{\pm,+}^{\prime
}\phi_{k}^{\pm}(I_{\pm})\bar{\Psi}_{k}^{\pm}(I_{\pm})\frac{1}{\omega_{\pm
}^{\prime}\left(  I_{\pm}\right)  }.\label{eqn-H-K}%
\end{equation}
Noting that $\mu_{\pm,\pm}^{\prime}\equiv0$ for $|e_{\pm}|\leq\sigma$, so
$H_{k}^{\pm,+}\left(  \omega_{\pm}\right)  =0$ for $\omega_{\pm}$ near $0$
(i.e. near separatrix). By zero extension, we can think of $H_{k}^{\pm
,+}\left(  \omega_{\pm}\right)  $ as a function defined in $\mathbf{R}%
$\textbf{. }Then the first integral in (\ref{integral-L2-norm}) can be
estimated by
\[
\left\vert \int_{\mathbf{R}}\frac{H_{k}^{\pm,+}\left(  \omega_{\pm}\right)
}{\omega_{\pm}+\frac{\lambda}{ik}}d\omega_{\pm}\right\vert \lesssim\left\Vert
H_{k}^{\pm,+}\right\Vert _{H^{1}\left(  \mathbf{R}\right)  },
\]
by Lemma \ref{lemma-hardy}. Since $\frac{\partial x}{\partial\theta}$ is
bounded \textit{outside }of separatrix, so
\begin{equation}
\left\vert \phi_{k}^{\pm}\left(  I_{\pm}\right)  \right\vert =\left\vert
\frac{1}{2\pi ki}\int_{0}^{2\pi}e^{-ik\theta_{\pm}}\phi^{\prime}%
(x)\frac{\partial x}{\partial\theta_{\pm}}d\theta_{\pm}\right\vert \leq
\frac{1}{2\pi k}||\phi||_{W^{1,1}}\lesssim\frac{1}{k}||\phi||_{H^{1}%
}.\label{estimate-phi-k}%
\end{equation}
Similarly, $\left\vert \bar{\Psi}_{k}^{\pm}\left(  I_{\pm}\right)  \right\vert
\lesssim\frac{1}{k}||\Psi||_{H^{1}}$. By assumption (\ref{condition-decay}),
\[
\int\sqrt{e_{\pm}}\left(  \mu_{\pm,+}^{\prime}\right)  ^{2}de_{\pm}<\infty,
\]
thus from
\[
\frac{d\omega_{\pm}}{de_{\pm}}=\frac{d\omega_{\pm}}{dI_{\pm}}\frac{dI_{\pm}%
}{de_{\pm}}=\frac{\omega_{\pm}^{\prime}}{\omega_{\pm}},
\]
and (\ref{expansion-w})
\[
\int\left(  \omega_{\pm}\mu_{\pm,+}^{\prime}\right)  ^{2}d\omega_{\pm}%
\lesssim\int\sqrt{e_{\pm}}\left(  \mu_{\pm,+}^{\prime}\right)  ^{2}de_{\pm
}<\infty
\]
Thanks to Lemma 8, we know that $\omega_{\pm}^{\prime}(I_{\pm})\backsim1$ when
$\mu_{\pm,\pm}^{\prime}\neq0$, since the supports of $\mu_{\pm,\pm}^{\prime}$
are outside the separatrix. So
\begin{align}
d\omega_{\pm}  & =\omega_{\pm}^{\prime}(I_{\pm})dI_{\pm}\backsim dI_{\pm
},\label{change}\\
\text{\ \ }\frac{d}{d\omega_{\pm}}  & =\frac{dI_{\pm}}{d\omega_{\pm}}\frac
{d}{dI_{\pm}}=\frac{1}{\omega_{\pm}^{\prime}(I_{\pm})}\frac{d}{dI_{\pm}%
}\backsim\frac{d}{dI_{\pm}},\nonumber
\end{align}
when $\mu_{\pm,\pm}^{\prime}\neq0$. Combining above, we get%
\[
\left\Vert H_{k}^{\pm,+}\right\Vert _{L^{2}\left(  \mathbf{R}\right)
}\lesssim\frac{1}{k^{2}}||\phi||_{H^{1}}||\Psi||_{H^{1}},
\]
where $H_{k}^{\pm,+}$ is defined in (\ref{eqn-H-K}). By using Lemma
\ref{epestimate} and (\ref{estimate-phi-k}), we have
\begin{align*}
&  \left\Vert \frac{d}{\omega_{\pm}}H_{k}^{\pm,+}\right\Vert _{L^{2}\left(
\mathbf{R}\right)  }\\
&  \lesssim\frac{1}{k}\left(  ||\phi||_{H^{1}}\left(  \int\left(  \omega_{\pm
}\mu_{\pm,+}^{\prime}\right)  ^{2}\left\vert \bar{\Psi}_{k}^{\pm\prime
}\right\vert ^{2}dI_{\pm}\right)  ^{\frac{1}{2}}+||\Psi||_{H^{1}}\left(
\int\left(  \omega_{\pm}\mu_{\pm,+}^{\prime}\right)  ^{2}\left\vert \phi
_{k}^{\pm\prime}\right\vert ^{2}dI_{\pm}\right)  ^{\frac{1}{2}}\right)  \\
&  \ \ \ \ \ \ \ +\frac{1}{k^{2}}||\phi||_{H^{1}}||\Psi||_{H^{1}}.
\end{align*}
Thus
\begin{align*}
&  \ \ \ \ \ \left\vert \sum_{k\neq0}\int_{\mathbf{R}}\frac{H_{k}^{\pm
,+}\left(  \omega_{\pm}\right)  }{\omega_{\pm}+\frac{\lambda}{ik}}d\omega
_{\pm}\right\vert \lesssim\sum_{k\neq0}\left\Vert H_{k}^{\pm,+}\right\Vert
_{H^{1}\left(  \mathbf{R}\right)  }\\
&  \lesssim\sum_{k\neq0}(\frac{1}{k}||\phi||_{H^{1}}\left(  \int\left(
\omega_{\pm}\mu_{\pm,+}^{\prime}\right)  ^{2}\left\vert \bar{\Psi}_{k}%
^{\pm\prime}\right\vert ^{2}dI_{\pm}\right)  ^{\frac{1}{2}}+\frac{1}{k}%
||\Psi||_{H^{1}}\left(  \int\left(  \omega_{\pm}\mu_{\pm,+}^{\prime}\right)
^{2}\left\vert \phi_{k}^{\pm\prime}\right\vert ^{2}dI_{\pm}\right)  ^{\frac
{1}{2}}\\
&  \ \ \ \ \ \ +\frac{1}{k^{2}}||\phi||_{H^{1}}||\Psi||_{H^{1}})\\
&  \lesssim||\phi||_{H^{1}}\left(  \sum_{k}\int\left(  \omega_{\pm}\mu_{\pm
,+}^{\prime}\right)  ^{2}\left\vert \bar{\Psi}_{k}^{\pm\prime}\right\vert
^{2}dI_{\pm}\right)  ^{\frac{1}{2}}+||\Psi||_{H^{1}}\left(  \sum_{k}%
\int\left(  \omega_{\pm}\mu_{\pm,+}^{\prime}\right)  ^{2}\left\vert \phi
_{k}^{\pm\prime}\right\vert ^{2}dI_{\pm}\right)  ^{\frac{1}{2}}\\
&  \ \ \ \ \ \ \ +||\phi||_{H^{1}}||\Psi||_{H^{1}}\\
&  \lesssim||\phi||_{H^{1}}||\Psi||_{H^{1}},
\end{align*}
by Lemma \ref{lemma-bound-phi-k}. The second term in (\ref{integral-L2-norm})
can be estimated in the same way. So we have
\[
\left\Vert \partial_{x}^{-2}\rho(\lambda,\varepsilon)\phi\right\Vert _{H^{1}%
}=\left\Vert \Phi\right\Vert _{H^{1}}\lesssim||\phi||_{H^{1}}.
\]

Proof of (iii): By (\ref{texpansion}), when $|e_{\pm}|\geq\sigma,$ we have
\[
\omega_{\pm}=2\pi/T_{\pm}(e_{\pm})\gtrsim\sqrt{\sigma}.
\]
So for $\left\vert \lambda\right\vert <<\sqrt{\sigma}$,
\begin{equation}
\left\vert \frac{\lambda}{ik}+\omega_{\pm}\right\vert \gtrsim\sqrt{\sigma
},\ \text{uniformly for any }|k|\geq1\text{.}\label{bound-omega}%
\end{equation}
Hence, the integrals in (\ref{integral-L2-norm}) are clearly bounded by
$||\phi||_{H^{1}}||\Psi||_{H^{1}}$. We further take complex$\ \lambda$
derivatives of (\ref{integral-L2-norm}) to get
\begin{align*}
&  \ \ \ \ \ (\partial_{\lambda}\left(  \rho(\lambda,\varepsilon)\phi\right)
,\Psi)\\
&  =-\sum_{k\neq0,\ \pm}\left(  -\int\frac{\omega_{\pm}\mu_{\pm,+}^{\prime}%
}{ik(\omega_{\pm}+\frac{\lambda}{ik})^{2}}\phi_{k}^{\pm}(I_{\pm})\bar{\Psi
}_{k}^{\pm}(I_{\pm})dI_{\pm}+\int\frac{\omega_{\pm}\mu_{\pm,-}^{\prime}%
}{\left(  \omega_{\pm}-\frac{\lambda}{ik}\right)  ^{2}}\phi_{k}^{\pm}\bar
{\Psi}_{k}^{\pm}dI_{\pm}\right)  ,
\end{align*}
which by (\ref{bound-omega}) again is bounded by $||\phi||_{H^{1}}%
||\Psi||_{H^{1}}$. This shows that $\partial_{\lambda}K\left(  \lambda
,\varepsilon\right)  $ is bounded operator in $H_{0}^{1}$ and thus $K\left(
\lambda,\varepsilon\right)  $ is an analytic operator in $H_{0}^{1}\ $when
$|\lambda|<<1$.
\end{proof}

\begin{proposition}
\label{prop-operator-difference}For $\varepsilon<<\sigma$ and any $\lambda
\in\mathbf{C}$ with $\operatorname{Re}\lambda>0\mathbf{,\,}$\
\begin{equation}
||K(\lambda,\varepsilon)-K(\lambda,0)||_{L(H_{0}^{1},H_{0}^{1})}\leq
C\sqrt{\varepsilon}, \label{estimate-difference-operator}%
\end{equation}
where $C$ is a positive constant independent of $\lambda$.
\end{proposition}

\begin{proof}
It is equivalent to show that
\[
\left\Vert \partial_{x}^{-2}\rho(\lambda,\varepsilon)-\partial_{x}^{-2}%
\rho(\lambda,0)\right\Vert _{L(H_{\varepsilon}^{1},H_{\varepsilon}^{1})}\leq
C\sqrt{\varepsilon},
\]
where $C>0$ is independent of $\lambda$. Let $\phi,\Psi\in H_{\varepsilon}%
^{1}$, then it suffices to show that
\[
\left\vert \left(  \left(  \rho(\lambda,\varepsilon)-\rho(\lambda,0)\right)
\phi,\Psi\right)  \right\vert \leq C\sqrt{\varepsilon}||\phi||_{H^{1}}%
||\Psi||_{H^{1}},
\]
for some constant $C$ independent of $\lambda$. 

The proof is split into three steps. 

\textit{Step 1}. \textit{Representation. }We note that for the homogeneous
case $\left(  \varepsilon=0\right)  \ $with period $P_{\beta}$, the
action-angle variables become
\[
\omega_{\pm}=\frac{2\pi}{P_{\beta}}|v|,\ I_{\pm}=\frac{P_{\beta}}{2\pi
}\left\vert v\right\vert ,\ \theta_{\pm}=\frac{2\pi}{P_{\beta}}x.
\]
Then
\begin{align*}
&  \ \ \ \ \left(  \rho(\lambda,0)\phi,\Psi\right)  \\
&  =\sum_{k\neq0,\ \pm}\left(  \int\frac{\omega_{\pm}\mu_{\pm,+}^{\prime}%
}{\omega_{\pm}+\frac{\lambda}{ik}}\phi_{k}^{\pm,0}\bar{\Psi}_{k}^{\pm
,0}dI_{\pm}+\int\frac{\omega_{\pm}\mu_{\pm,-}^{\prime}}{\omega_{\pm}%
-\frac{\lambda}{ik}}\phi_{k}^{\pm,0}\bar{\Psi}_{k}^{\pm,0}dI_{\pm}\right)  \\
&  =\sum_{k\neq0,\ \pm}\left(  \frac{P_{\beta}}{2\pi}\right)  ^{2}\left(
\int\frac{\omega_{\pm}\mu_{\pm,+}^{\prime}}{\omega_{\pm}+\frac{\lambda}{ik}%
}\phi_{k}^{\pm,0}\bar{\Psi}_{k}^{\pm,0}d\omega_{\pm}+\int\frac{\omega_{\pm}%
\mu_{\pm,-}^{\prime}}{\omega_{\pm}-\frac{\lambda}{ik}}\phi_{k}^{\pm,0}%
\bar{\Psi}_{k}^{\pm,0}d\omega_{\pm}\right)  \\
&  =\sum_{k\neq0,\ \pm}\int\frac{H_{k}^{\pm,+;0}}{\omega_{\pm}+\frac{\lambda
}{ik}}d\omega_{\pm}+\int\frac{H_{k}^{\pm,-;0}}{\omega_{\pm}-\frac{\lambda}%
{ik}}d\omega_{\pm},
\end{align*}
where%
\[
\mu_{\pm,\pm}^{\prime}=\mu_{\pm,\pm}^{\prime}\left(  \frac{1}{2}\left(
\frac{P_{\beta}}{2\pi}\right)  ^{2}\omega_{\pm}^{2}\right)  =\mu_{\pm,\pm
}^{\prime}\left(  \frac{1}{2}v^{2}\right)  ,
\]%
\begin{align*}
\phi_{k}^{\pm,0} &  \equiv\frac{1}{2\pi}\int_{0}^{2\pi}e^{-ik\theta_{\pm}}%
\phi(\frac{P_{\beta}}{2\pi}\theta_{\pm})d\theta_{\pm},\ \\
\Psi_{k}^{\pm,0} &  \equiv\frac{1}{2\pi}\int_{0}^{2\pi}e^{-ik\theta_{\pm}}%
\Psi(\frac{P_{\beta}}{2\pi}\theta_{\pm})d\theta_{\pm},
\end{align*}
and
\begin{align*}
H_{k}^{\pm,+;0}\left(  \omega_{\pm}\right)   &  =\left(  \frac{P_{\beta}}%
{2\pi}\right)  ^{2}\omega_{\pm}\mu_{\pm,+}^{\prime}\left(  \frac{1}{2}\left(
\frac{P_{\beta}}{2\pi}\right)  ^{2}\omega_{\pm}^{2}\right)  \phi_{k}^{\pm
,0}\bar{\Psi}_{k}^{\pm,0},\\
H_{k}^{\pm,-;0}\left(  \omega_{\pm}\right)   &  =\left(  \frac{P_{\beta}}%
{2\pi}\right)  ^{2}\omega_{\pm}\mu_{\pm,-}^{\prime}\left(  \frac{1}{2}\left(
\frac{P_{\beta}}{2\pi}\right)  ^{2}\omega_{\pm}^{2}\right)  \phi_{k}^{\pm
,0}\bar{\Psi}_{k}^{\pm,0}.
\end{align*}
In the case $\varepsilon>0$, we use the same notations $\left(  I_{\pm
},\ \theta_{\pm}\right)  $ and $\omega_{\pm}\ $for action-angle variables and
frequency, to make it convient to estimate the difference $\left(  \left(
\rho(\lambda,\varepsilon)-\rho(\lambda,0)\right)  \phi,\Psi\right)  $. Then as
in the proof of Lemma \ref{Lemma-operator-analytic}, we can write
\[
\left(  \rho(\lambda,\varepsilon)\phi,\Psi\right)  =\sum_{k\neq0,\ \pm}%
\int\frac{H_{k}^{\pm,+;\varepsilon}\left(  \omega_{\pm}\right)  }{\omega_{\pm
}+\frac{\lambda}{ik}}d\omega_{\pm}+\int\frac{H_{k}^{\pm,-;\varepsilon}\left(
\omega_{\pm}\right)  }{\omega_{\pm}-\frac{\lambda}{ik}}d\omega_{\pm},
\]
where%
\begin{align*}
H_{k}^{\pm,+;\varepsilon}\left(  \omega_{\pm}\right)   &  =\omega_{\pm}%
\mu_{\pm,+}^{\prime}\left(  e_{\pm}\right)  \phi_{k}^{\pm,\varepsilon}(I_{\pm
})\bar{\Psi}_{k}^{\pm,\varepsilon}(I_{\pm})\frac{1}{\omega_{\pm}^{\prime
}\left(  I_{\pm}\right)  },\\
H_{k}^{\pm,-;\varepsilon}\left(  \omega_{\pm}\right)   &  =\omega_{\pm}%
\mu_{\pm,-}^{\prime}\left(  e_{\pm}\right)  \phi_{k}^{\pm,\varepsilon}(I_{\pm
})\bar{\Psi}_{k}^{\pm,\varepsilon}(I_{\pm})\frac{1}{\omega_{\pm}^{\prime
}\left(  I_{\pm}\right)  },
\end{align*}
with%
\begin{align*}
\phi_{k}^{\pm,\varepsilon}(I_{\pm}) &  =\frac{1}{2\pi}\int_{0}^{2\pi}%
\phi(x\left(  I_{\pm},\theta_{\pm}\right)  )e^{-ik\theta_{\pm}}d\theta_{\pm
},\\
\Psi_{k}^{\pm,\varepsilon}(I_{\pm}) &  =\frac{1}{2\pi}\int_{0}^{2\pi}%
\Psi(x\left(  I_{\pm},\theta_{\pm}\right)  )e^{-ik\theta_{\pm}}d\theta_{\pm}.
\end{align*}
Here, in the above formula for $\phi_{k}^{\pm,\varepsilon},\Psi_{k}%
^{\pm,\varepsilon}$, we use $\left(  x\left(  I_{\pm},\theta_{\pm}\right)
,v\left(  I_{\pm},\theta_{\pm}\right)  \right)  $ to denote the action-angle
transform $\left(  I_{\pm},\theta_{\pm}\right)  \rightarrow\left(  x,v\right)
$ in the case $\varepsilon>0$. Defining
\begin{align}
G_{k}^{\pm,+}\left(  \omega_{\pm}\right)   &  =H_{k}^{\pm,+;\varepsilon
}\left(  \omega_{\pm}\right)  -H_{k}^{\pm,+;0}\left(  \omega_{\pm}\right)
,\label{defn-G-k}\\
G_{k}^{\pm,-}\left(  \omega_{\pm}\right)   &  =H_{k}^{\pm,-;\varepsilon
}\left(  \omega_{\pm}\right)  -H_{k}^{\pm,-;0}\left(  \omega_{\pm}\right)
,\nonumber
\end{align}
then we get
\[
\left(  \left(  \rho(\lambda,\varepsilon)-\rho(\lambda,0)\right)  \phi
,\Psi\right)  =\sum_{k\neq0,\ \pm}\int\frac{G_{k}^{\pm,+}\left(  \omega_{\pm
}\right)  }{\omega_{\pm}+\frac{\lambda}{ik}}d\omega_{\pm}+\int\frac{G_{k}%
^{\pm,-}\left(  \omega_{\pm}\right)  }{\omega_{\pm}-\frac{\lambda}{ik}}%
d\omega_{\pm}.
\]
By Lemma \ref{lemma-hardy}, the proof is reduced to estimate $\left\Vert
G_{k}^{\pm,\pm}\right\Vert _{H^{1}}$. We write%
\begin{align}
G_{k}^{\pm,+}\left(  \omega_{\pm}\right)   &  =\omega_{\pm}\left(  \mu_{\pm
,+}^{\prime}\left(  e_{\pm}\right)  -\mu_{\pm,+}^{\prime}\left(  \frac{1}%
{2}\left(  \frac{P_{\beta}}{2\pi}\right)  ^{2}\omega_{\pm}^{2}\right)
\right)  \phi_{k}^{\pm,\varepsilon}(I_{\pm})\bar{\Psi}_{k}^{\pm,\varepsilon
}(I_{\pm})\frac{1}{\omega_{\pm}^{\prime}\left(  I_{\pm}\right)  }%
\label{defn-G-K-j}\\
&  \ \ \ \ \ +\omega_{\pm}\mu_{\pm,+}^{\prime}\left(  \frac{1}{2}\left(
\frac{P_{\beta}}{2\pi}\right)  ^{2}\omega_{\pm}^{2}\right)  \left(  \frac
{1}{\omega_{\pm}^{\prime}\left(  I_{\pm}\right)  }-\left(  \frac{P_{\beta}%
}{2\pi}\right)  ^{2}\right)  \phi_{k}^{\pm,\varepsilon}(I_{\pm})\bar{\Psi}%
_{k}^{\pm,\varepsilon}(I_{\pm})\nonumber\\
&  \ \ \ \ \ +\omega_{\pm}\mu_{\pm,+}^{\prime}\left(  \frac{1}{2}\left(
\frac{P_{\beta}}{2\pi}\right)  ^{2}\omega_{\pm}^{2}\right)  \left(
\frac{P_{\beta}}{2\pi}\right)  ^{2}\left(  \phi_{k}^{\pm,\varepsilon}(I_{\pm
})-\phi_{k}^{\pm,0}\right)  \bar{\Psi}_{k}^{\pm,\varepsilon}(I_{\pm
})\nonumber\\
&  \ \ \ \ \ +\omega_{\pm}\mu_{\pm,+}^{\prime}\left(  \frac{1}{2}\left(
\frac{P_{\beta}}{2\pi}\right)  ^{2}\omega_{\pm}^{2}\right)  \left(
\frac{P_{\beta}}{2\pi}\right)  ^{2}\phi_{k}^{\pm,0}\left(  \bar{\Psi}_{k}%
^{\pm,\varepsilon}(I_{\pm})-\bar{\Psi}_{k}^{\pm,0}\right)  \nonumber\\
&  =\sum_{j=1}^{4}G_{k}^{\pm,+;j}\left(  \omega_{\pm}\right)  .\nonumber
\end{align}

\textit{Step 2.} \textit{Estimates for }$\phi_{k}^{\pm,\varepsilon}-\phi
_{k}^{\pm,0}$ \textit{and }$\psi_{k}^{\pm,\varepsilon}-\psi_{k}^{\pm,0}$
\textit{when }$|e_{\pm}|\geq\frac{\sigma}{2}.$

We note that
\begin{align}
\phi_{k}^{\pm,\varepsilon}(I_{\pm})-\phi_{k}^{\pm,0} &  =\frac{1}{2\pi}%
\int_{0}^{2\pi}e^{-ik\theta_{\pm}}[\phi(x\left(  I_{\pm},\theta_{\pm}\right)
)-\phi(\frac{P_{\beta}}{2\pi}\theta_{\pm})]d\theta_{\pm}%
\label{expansion-phi-difference}\\
&  =\frac{1}{2\pi}\int_{0}^{2\pi}e^{-ik\theta_{\pm}}\left[  \int%
_{\frac{P_{\beta}}{2\pi}\theta_{\pm}}^{x\left(  I_{\pm},\theta_{\pm}\right)
}\phi_{x}(\zeta)d\zeta\right]  d\theta_{\pm},\nonumber
\end{align}
and $\left\vert x-\frac{P_{\beta}}{2\pi}\theta_{\pm}\right\vert \lesssim
\varepsilon$ by Lemma \ref{epestimate}, so
\begin{equation}
\left\vert \phi_{k}^{\pm,\varepsilon}(I_{\pm})-\phi_{k}^{\pm,0}\right\vert
\leq\max_{\theta_{\pm}}\left\vert \int_{\frac{P_{\beta}}{2\pi}\theta_{\pm}%
}^{x\left(  I_{\pm},\theta_{\pm}\right)  }\phi_{x}(\zeta)d\zeta\right\vert
\leq\min\left\{  \sqrt{\varepsilon}||\phi||_{H^{1}},\varepsilon||\phi
_{x}||_{\infty}\right\}  ,\label{maximum-phi-difference}%
\end{equation}
and
\begin{align}
&  \ \ \ \ \sum_{k}\int\left(  \omega_{\pm}\mu_{\pm,+}^{\prime}\right)
^{2}\left\vert \phi_{k}^{\pm,\varepsilon}(I_{\pm})-\phi_{k}^{\pm,0}\right\vert
^{2}dI_{\pm}\label{estimate-phi-difference-L^2}\\
&  =\int\int\left(  \omega_{\pm}\mu_{\pm,+}^{\prime}\right)  ^{2}\left\vert
\int_{\frac{P_{\beta}}{2\pi}\theta_{\pm}\left(  x,v\right)  }^{x}\phi
_{x}(\zeta)d\zeta\right\vert ^{2}dxdv\nonumber\\
&  \leq\min\left\{  \varepsilon||\phi||_{H^{1}}^{2},\varepsilon^{2}||\phi
_{x}||_{\infty}^{2}\right\}  \int\int\left(  \omega_{\pm}\mu_{\pm,+}^{\prime
}\right)  ^{2}dxdv\nonumber\\
&  \lesssim\min\left\{  \varepsilon||\phi||_{H^{1}}^{2},\varepsilon^{2}%
||\phi_{x}||_{\infty}^{2}\right\}  .\nonumber
\end{align}
Similarly, we have
\[
\Psi_{k}^{\pm,\varepsilon}(I_{\pm})-\Psi_{k}^{\pm,0}=\frac{1}{2\pi}\int%
_{0}^{2\pi}e^{-ik\theta_{\pm}}\left[  \int_{\frac{P_{\beta}}{2\pi}\theta_{\pm
}}^{x\left(  I_{\pm},\theta_{\pm}\right)  }\Psi_{x}(\zeta)d\zeta\right]
d\theta_{\pm},
\]%
\begin{equation}
\left\vert \Psi_{k}^{\pm,\varepsilon}(I_{\pm})-\Psi_{k}^{\pm,0}\right\vert
\leq\min\left\{  \sqrt{\varepsilon}\left\Vert \Psi\right\Vert _{H^{1}%
},\varepsilon\left\Vert \Psi_{x}\right\Vert _{\infty}\right\}
,\label{maximum-psi-difference}%
\end{equation}
and
\begin{equation}
\sum_{k}\int\left(  \omega_{\pm}\mu_{\pm,+}^{\prime}\right)  ^{2}\left\vert
\Psi_{k}^{\pm,\varepsilon}(I_{\pm})-\Psi_{k}^{\pm,0}\right\vert ^{2}dI_{\pm
}\lesssim\min\left\{  \varepsilon\left\Vert \Psi\right\Vert _{H^{1}}%
^{2},\varepsilon^{2}\left\Vert \Psi_{x}\right\Vert _{\infty}^{2}\right\}
.\label{estimate-psi-difference-L^2}%
\end{equation}
This completes Step 2.

\textit{Step 3.} \textit{Estimate for }$||G^{\pm,\pm}||_{H^{1}}.$

Here, the function $G^{\pm,\pm}$ is defined in (\ref{defn-G-k}). We will
estimate $||G^{\pm,+}||_{H^{1}}$ and it is similar for $||G^{\pm,-}||_{H^{1}}%
$. we have
\[
e_{\pm}-\frac{1}{2}\left(  \frac{P_{\beta}}{2\pi}\right)  ^{2}\omega_{\pm}%
^{2}=O(\varepsilon),
\]
and thus by (\ref{condition-decay})
\[
\left\vert \mu_{\pm,+}^{\prime}\left(  e_{\pm}\right)  -\mu_{\pm,+}^{\prime
}\left(  \frac{1}{2}\left(  \frac{P_{\beta}}{2\pi}\right)  ^{2}\omega_{\pm
}^{2}\right)  \right\vert \leq C\varepsilon\left(  1+\omega_{\pm}^{2}\right)
^{-\gamma},\ \gamma>1.
\]
We recall that $G^{\pm,+}=\sum_{j=1}^{4}G^{\pm,+;j}$ (see (\ref{defn-G-K-j}))
in the following estimates, where we also use (\ref{change}) to transform the
integrals $\int\cdots d\omega_{\pm}$ to $\int\cdots dI_{\pm}$. By
(\ref{estimate-phi-k}),
\[
\left\Vert G_{k}^{\pm,+;1}\right\Vert _{L^{2}}\lesssim\frac{\varepsilon}%
{k^{2}}||\phi||_{H^{1}}||\Psi||_{H^{1}}.
\]
By Lemma \ref{epestimate} $\left\vert \frac{1}{\omega_{\pm}^{\prime}\left(
I_{\pm}\right)  }-\left(  \frac{P_{\beta}}{2\pi}\right)  ^{2}\right\vert
\lesssim\varepsilon$, so
\[
\left\Vert G_{k}^{\pm,+;2}\right\Vert _{L^{2}}\lesssim\frac{\varepsilon}%
{k^{2}}||\phi||_{H^{1}}||\Psi||_{H^{1}}.
\]
It is strightforward to get
\[
\left\Vert G_{k}^{\pm,+;3}\right\Vert _{L^{2}}\lesssim\frac{1}{k}%
||\Psi||_{H^{1}}\left(  \int\left(  \omega_{\pm}\mu_{\pm,+}^{\prime}\right)
^{2}\left\vert \phi_{k}^{\pm,\varepsilon}(I_{\pm})-\phi_{k}^{\pm,0}\right\vert
^{2}dI_{\pm}\right)  ^{\frac{1}{2}},
\]
and
\[
\left\Vert G_{k}^{\pm,+;4}\right\Vert _{L^{2}}\lesssim\frac{1}{k}%
||\phi||_{H^{1}}\left(  \int\left(  \omega_{\pm}\mu_{\pm,+}^{\prime}\right)
^{2}\left\vert \Psi_{k}^{\pm,\varepsilon}(I_{\pm})-\Psi_{k}^{\pm,0}\right\vert
^{2}dI_{\pm}\right)  ^{\frac{1}{2}}.
\]
Now we estimate $\left\Vert \frac{d}{d\omega_{\pm}}G_{k}^{\pm,+}\left(
\omega_{\pm}\right)  \right\Vert _{L^{2}}$. We note that
\begin{align*}
&  \ \ \ \left\vert \frac{d}{d\omega_{\pm}}\left[  \mu_{\pm,+}^{\prime}\left(
e_{\pm}\right)  -\mu_{\pm,+}^{\prime}\left(  \frac{1}{2}\left(  \frac
{P_{\beta}}{2\pi}\right)  ^{2}\omega_{\pm}^{2}\right)  \right]  \right\vert \\
&  =\left\vert \mu_{\pm,+}^{\prime\prime}\left(  e_{\pm}\right)  \frac
{\omega_{\pm}}{\omega_{\pm}^{\prime}}-\mu_{\pm,+}^{\prime\prime}\left(
\frac{1}{2}\left(  \frac{P_{\beta}}{2\pi}\right)  ^{2}\omega_{\pm}^{2}\right)
\left(  \frac{P_{\beta}}{2\pi}\right)  ^{2}\omega_{\pm}\right\vert \\
&  \lesssim\varepsilon\omega_{\pm}\left(  1+\omega_{\pm}^{2}\right)
^{-\gamma},
\end{align*}
by (\ref{condition-decay}) and Lemma \ref{epestimate}. So we can get
\begin{align*}
&  \ \ \ \ \left\Vert \frac{d}{d\omega_{\pm}}G_{k}^{\pm,+;1}\left(
\omega_{\pm}\right)  \right\Vert _{L^{2}}\\
&  \lesssim\frac{\varepsilon}{k^{2}}||\phi||_{H^{1}}||\Psi||_{H^{1}}+\frac
{1}{k}||\phi||_{H^{1}}\left(  \int\left(  \omega_{\pm}\mu_{\pm,+}^{\prime
}\right)  ^{2}\left\vert \left(  \Psi_{k}^{\pm,\varepsilon}\right)  ^{\prime
}(I_{\pm})\right\vert ^{2}dI_{\pm}\right)  ^{\frac{1}{2}}\\
&  \ \ \ \ \ +\frac{1}{k}||\Psi||_{H^{1}}\left(  \int\left(  \omega_{\pm}%
\mu_{\pm,+}^{\prime}\right)  ^{2}\left\vert \left(  \phi_{k}^{\pm,\varepsilon
}\right)  ^{\prime}(I_{\pm})\right\vert ^{2}dI_{\pm}\right)  ^{\frac{1}{2}}.
\end{align*}
For $\left\Vert \frac{d}{d\omega_{\pm}}G_{k}^{\pm,+;2}\left(  \omega_{\pm
}\right)  \right\Vert _{L^{2}},$ the above estimate also holds true by using
\[
\left\vert \frac{d}{dI_{\pm}}\left(  \frac{1}{\omega_{\pm}^{\prime}(I_{\pm}%
)}\right)  \right\vert \lesssim\varepsilon.
\]
By using (\ref{maximum-phi-difference}) and (\ref{maximum-psi-difference}), we
can get that
\begin{align*}
&  \ \ \left\Vert \frac{d}{d\omega_{\pm}}G_{k}^{\pm,+;3}\left(  \omega_{\pm
}\right)  \right\Vert _{L^{2}}+\left\Vert \frac{d}{d\omega_{\pm}}G_{k}%
^{\pm,+;4}\left(  \omega_{\pm}\right)  \right\Vert _{L^{2}}\\
&  \lesssim\frac{1}{k}||\phi||_{H^{1}}\left(  \int\left(  \omega_{\pm}\mu
_{\pm,+}^{\prime}\right)  ^{2}\left\vert \left(  \Psi_{k}^{\pm,\varepsilon
}\right)  ^{\prime}(I_{\pm})\right\vert ^{2}dI_{\pm}\right)  ^{\frac{1}{2}}\\
&  \ \ \ \ \ \ +\frac{1}{k}||\Psi||_{H^{1}}\left(  \int\left(  \omega_{\pm}%
\mu_{\pm,+}^{\prime}\right)  ^{2}\left\vert \left(  \phi_{k}^{\pm,\varepsilon
}\right)  ^{\prime}(I_{\pm})\right\vert ^{2}dI_{\pm}\right)  ^{\frac{1}{2}}\\
&  \ \ \ \ \ \ +\min\left\{  \sqrt{\varepsilon}||\Psi||_{H^{1}},\varepsilon
||\Psi_{x}||_{\infty}\right\}  \left(  \int\left(  \omega_{\pm}^{2}\mu_{\pm
,+}^{\prime\prime}\right)  ^{2}\left\vert \phi_{k}^{\pm,\varepsilon
}\right\vert ^{2}dI_{\pm}\right)  ^{\frac{1}{2}}\\
&  \ \ \ \ \ \ +\min\left\{  \sqrt{\varepsilon}||\phi||_{H^{1}},\varepsilon
||\phi_{x}||_{\infty}\right\}  \left(  \int\left(  \omega_{\pm}^{2}\mu_{\pm
,+}^{\prime\prime}\right)  ^{2}\left\vert \Psi_{k}^{\pm,\varepsilon
}\right\vert ^{2}dI_{\pm}\right)  ^{\frac{1}{2}}.
\end{align*}
Combining above, we have
\begin{align*}
\sum_{k}\left\Vert G_{k}^{\pm,+}\right\Vert _{H^{1}} &  \lesssim\sum
_{k}\{\frac{\varepsilon}{k^{2}}||\phi||_{H^{1}}||\Psi||_{H^{1}}+\frac{1}%
{k}||\Psi||_{H^{1}}\left(  \int\left(  \omega_{\pm}\mu_{\pm,+}^{\prime
}\right)  ^{2}\left\vert \phi_{k}^{\pm,\varepsilon}(I_{\pm})-\phi_{k}^{\pm
,0}\right\vert ^{2}dI_{\pm}\right)  ^{\frac{1}{2}}\\
&  \ \ \ \ +\frac{1}{k}||\phi||_{H^{1}}\left(  \int\left(  \omega_{\pm}%
\mu_{\pm,+}^{\prime}\right)  ^{2}\left\vert \Psi_{k}^{\pm,\varepsilon}(I_{\pm
})-\Psi_{k}^{\pm,0}\right\vert ^{2}dI_{\pm}\right)  ^{\frac{1}{2}}\\
&  \ \ \ \ +\frac{1}{k}||\phi||_{H^{1}}\left(  \int\left(  \omega_{\pm}%
\mu_{\pm,+}^{\prime}\right)  ^{2}\left\vert \left(  \Psi_{k}^{\pm,\varepsilon
}\right)  ^{\prime}(I_{\pm})\right\vert ^{2}dI_{\pm}\right)  ^{\frac{1}{2}}\\
&  \ \ \ \ +\frac{1}{k}||\Psi||_{H^{1}}\left(  \int\left(  \omega_{\pm}%
\mu_{\pm,+}^{\prime}\right)  ^{2}\left\vert \left(  \phi_{k}^{\pm,\varepsilon
}\right)  ^{\prime}(I_{\pm})\right\vert ^{2}dI_{\pm}\right)  ^{\frac{1}{2}}\\
&  \ \ \ \ +\min\left\{  \sqrt{\varepsilon}||\Psi||_{H^{1}},\varepsilon
||\Psi_{x}||_{\infty}\right\}  \left(  \int\left(  \omega_{\pm}^{2}\mu_{\pm
,+}^{\prime\prime}\right)  ^{2}\left\vert \phi_{k}^{\pm,\varepsilon
}\right\vert ^{2}dI_{\pm}\right)  ^{\frac{1}{2}}\\
&  \ \ \ \ +\min\left\{  \sqrt{\varepsilon}||\phi||_{H^{1}},\varepsilon
||\phi_{x}||_{\infty}\right\}  \left(  \int\left(  \omega_{\pm}^{2}\mu_{\pm
,+}^{\prime\prime}\right)  ^{2}\left\vert \Psi_{k}^{\pm,\varepsilon
}\right\vert ^{2}dI_{\pm}\right)  ^{\frac{1}{2}}\}\\
&  \leq\sqrt{\varepsilon}||\Psi||_{H^{1}}||\phi||_{H^{1}},
\end{align*}
by using Lemma \ref{lemma-bound-phi-k} and (\ref{estimate-phi-difference-L^2}%
), (\ref{estimate-psi-difference-L^2}). Similarly, we get
\[
\sum_{k}\left\Vert G_{k}^{\pm,-}\right\Vert _{H^{1}}\lesssim\sqrt{\varepsilon
}||\Psi||_{H^{1}}||\phi||_{H^{1}},
\]
So by Lemma \ref{lemma-hardy}, we have
\begin{align*}
\left\vert \left(  \left(  \rho(\lambda,\varepsilon)-\rho(\lambda,0)\right)
\phi,\Psi\right)  \right\vert  &  \lesssim\sum_{k}\left\Vert G_{k}^{\pm
,+}\right\Vert _{H^{1}}+\left\Vert G_{k}^{\pm,-}\right\Vert _{H^{1}}\\
&  \lesssim\sqrt{\varepsilon}||\Psi||_{H^{1}}||\phi||_{H^{1}}.
\end{align*}
\newline This finishes the proof.
\end{proof}

\begin{lemma}
\label{lemma-spectra-away-zero}Assume the conditions in Theorems
\ref{thm-main-even} and \ref{thm-main-uneven}. For any $\delta_{0}>0$, when
$\varepsilon$ is small enough, there exists no unstable eigenvalue $\lambda$
of the linearized VP (\ref{vlasoveigen}) with $\operatorname{Re}\lambda>0$ and
$\left\vert \lambda\right\vert \geq\delta_{0}$.
\end{lemma}

\begin{proof}
We note that $I+\partial_{x}^{-2}\rho(\lambda,0)$ is uniformly invertible when
$\operatorname{Re}\lambda>0$ and $\left\vert \lambda\right\vert \geq\delta
_{0}$ (see Lemma \ref{lemma-spectra-homogeneous}). So the conclusion follows
from Proposition \ref{prop-operator-difference}.
\end{proof}

The following lemma is used in the proof of Theorem \ref{thm-main-even}.

\begin{lemma}
\label{realderivative}Assume $\mu_{\pm,+}=\mu_{\pm,-}=\mu_{\pm}$ (even case).
Let $\lambda\in\mathbf{C}$ and $\left\vert \lambda\right\vert <<\sqrt{\sigma
}\mathbf{,}$ $\phi,\psi\in H_{0}^{1}$ and $\phi_{x},\psi_{x}\in L^{\infty}.$
Then
\[
\left\vert (\partial_{\lambda}\left(  K(\lambda,\varepsilon)-K(\lambda
,0)\right)  \phi,\psi)\right\vert \leq C\varepsilon\left\vert \lambda
\right\vert \left\Vert \phi_{x}\right\Vert _{L^{\infty}}\left\Vert
\psi\right\Vert _{L^{\infty}}.
\]

\end{lemma}

\begin{proof}
It suffices to estimate
\[
\left\vert (\partial_{\lambda}\left(  \partial_{x}^{-2}\rho(\lambda
,\varepsilon)-\partial_{x}^{-2}\rho(\lambda,0)\right)  \phi,\psi)\right\vert
\]
for $\phi,\psi\in H_{\varepsilon}^{1}$. Note that in the untrapped region, the
points $\left(  x,v\right)  $ and $\left(  x,-v\right)  $ have the same
action-angle coordinates $\left(  I_{\pm},\theta_{\pm}\right)  $. So when
$\mu_{\pm,-}=$ $\mu_{\pm,-}=\mu_{\pm}$, we can combine the integrals for
$\rho(\lambda,\varepsilon)\phi$ in (\ref{k}) as
\begin{align}
&  \sum_{k\neq0,\ \pm}\left[  \int_{v>0}\left(  \frac{\omega_{\pm}}%
{\omega_{\pm}+\frac{\lambda}{ik}}+\frac{\omega_{\pm}}{\omega_{\pm}%
-\frac{\lambda}{ik}}\right)  \mu_{_{\pm}}^{\prime}(e_{\pm})\phi_{k}^{\pm
}(I_{\pm})e^{ik\theta_{\pm}}dv\right]  \label{density-lambda-square}\\
&  =\sum_{k\neq0,\ \pm}\int_{v>0}\frac{2\omega_{\pm}^{2}\mu_{\pm}^{\prime
}(e_{\pm})}{\omega_{\pm}^{2}+\frac{\lambda^{2}}{k^{2}}}\phi_{k}^{\pm}(I_{\pm
})e^{ik\theta_{\pm}}dv,\nonumber
\end{align}
\qquad and then
\[
\partial_{\lambda}\rho(\lambda,\varepsilon)\phi=-4\sum_{k,\pm}\int_{v>0}%
\frac{\omega_{\pm}^{2}\mu_{\pm}^{\prime}(e_{\pm})\frac{\lambda}{k^{2}}%
}{\left(  \omega_{\pm}^{2}+\frac{\lambda^{2}}{k^{2}}\right)  ^{2}}\phi
_{k}^{\pm}(I_{\pm})e^{ik\theta_{\pm}}dv.
\]
Note that,
\[
(\partial_{x}^{-2}\partial_{\lambda}\rho(\lambda,\varepsilon)\phi
,\psi)=(\partial_{\lambda}\rho(\lambda,\varepsilon)\phi,\Psi),
\]
where $\Psi=\partial_{x}^{-2}\psi$. Then
\begin{align*}
&  \ \ \ \ (\partial_{\lambda}\left(  \partial_{x}^{-2}\rho(\lambda
,\varepsilon)-\partial_{x}^{-2}\rho(\lambda,0)\right)  \phi,\psi)\\
&  =(\partial_{\lambda}\left(  \rho(\lambda,\varepsilon)-\rho(\lambda
,0)\right)  \phi,\Psi)\\
&  =-4\lambda\sum_{k,\pm}\frac{1}{k^{2}}\int\frac{\omega_{\pm}^{2}}%
{[\omega_{\pm}^{2}+\frac{\lambda^{2}}{k^{2}}]^{2}}G_{k}^{\pm}(\omega_{\pm
})d\omega_{\pm},
\end{align*}
where
\[
G_{k}^{\pm}(\omega_{\pm})=G_{k}^{\pm,+}(\omega_{\pm})=G_{k}^{\pm,-}%
(\omega_{\pm})
\]
is defined in (\ref{defn-G-k}). Note that $\omega_{\pm}\gtrsim\sqrt{\sigma}$
in the support of $G_{k}^{\pm}$ and $\left\vert \lambda\right\vert
<<\sqrt{\sigma}$, so
\[
\left\Vert \frac{\omega_{\pm}^{2}}{[\omega_{\pm}^{2}+\frac{\lambda^{2}}{k^{2}%
}]^{2}}\right\Vert _{L^{2}}\leq C_{0}\text{ (independent of }\lambda).
\]
Thus
\begin{align*}
&  \left\vert (\partial_{\lambda}\left(  \rho(\lambda,\varepsilon
)-\rho(\lambda,0)\right)  \phi,\Psi)\right\vert \leq\sum_{k,\pm}%
\frac{4\left\vert \lambda\right\vert }{k^{2}}\int\frac{\omega_{\pm}^{2}%
}{[\omega_{\pm}^{2}+\frac{\lambda^{2}}{k^{2}}]^{2}}G_{k}^{\pm}(\omega_{\pm
})d\omega_{\pm}\\
\  &  \lesssim\sum_{k,\pm}\frac{4\left\vert \lambda\right\vert }{k^{2}%
}\left\Vert G_{k}^{\pm}\right\Vert _{L^{2}}\lesssim\varepsilon\left\vert
\lambda\right\vert \left\Vert \phi_{x}\right\Vert _{L^{\infty}}||\Psi
_{x}||_{L^{\infty}},
\end{align*}
by the same estimates for $\sum_{k}$ $\left\Vert G_{k}^{\pm,\pm}\right\Vert
_{L^{2}}$ as in the proof of Proposition \ref{prop-operator-difference} and by
choosing $L^{\infty}$ bounds in (\ref{estimate-phi-difference-L^2}) and
(\ref{estimate-psi-difference-L^2}). This finishes the proof of the lemma.
\end{proof}

\section{Proof of Main Theorems}

First, we study the spectrum of the linearized operator $\mathbf{I}%
+K(\lambda,0)\ $in $H_{0}^{1}$\ for flat homogeneous states.

\begin{lemma}
\label{lemma-spectra-homogeneous}

(i) (Even Case) Consider the flat homogeneous states satisfying the
assumptions in Theorem \ref{thm-main-even}. Then:

\ \ \ (1) $\mathbf{I}+K(\lambda,0)$ is invertible when $\operatorname{Re}%
\lambda>0$; for any $\delta>0,\ \left(  \mathbf{I}+K(\lambda,0)\right)  ^{-1}$
is uniformly bounded in $\left\{  \operatorname{Re}\lambda>0,\left\vert
\lambda\right\vert \geq\delta\right\}  $;

\ \ \ (2) $\mathbf{I}+K(\lambda,0)$ is analytic near $\lambda=0$; $\ker\left(
\mathbf{I}+K(0,0)\right)  =\{e^{i\frac{2\pi}{P_{0}}x},e^{-i\frac{2\pi}{P_{0}%
}x}\}$ $=\left\{  r_{1},r_{2}\right\}  $; $\mathbf{I}+K(0,0):$ $($\textbf{$I$%
}$-\mathbf{P})X\rightarrow($\textbf{$I$}$-\mathbf{P})X\ $is invertible$,$
where $\mathbf{P}$ is the project to span$\{r_{1},r_{2}\}.$

\ \ \ (3) $\det(\left(  \mathbf{I}+K(\lambda,0)\right)  r_{j},r_{i}%
)\backsim\lambda^{4}$ near $\lambda=0.$

(ii) (Uneven Case) Consider the flat homogeneous states satisfying the
assumptions in Theorem \ref{thm-main-uneven}. Then (1) and (2) are valid, and
$\det(\left(  \mathbf{I}+K(\lambda,0)\right)  r_{j},r_{i})\backsim\lambda^{2}$
near $\lambda=0.$
\end{lemma}

\begin{proof}
(i) (\textit{Even Case) }We only consider the two-species case and the proof
for the fixed ion case is similar. Denote $\mu_{\pm,+}=\mu_{\pm,-}=\mu_{\pm}$.

Proof of (1): We note that when $\varepsilon=0,$ $\mathbf{I}+K(\lambda,0)$ is
the Penrose diagonal operator%
\begin{align}
\left(  \mathbf{I}+K(\lambda,0)\right)  e^{ik\frac{2\pi}{P_{0}}x}  &
=\{\mathbf{I}+\partial_{x}^{-2}\rho(\lambda,0)\}e^{ik\frac{2\pi}{P_{0}}%
x}\label{penrose-operator}\\
&  =\left\{  1-\frac{4\pi^{2}}{P_{0}^{2}k^{2}}\int\frac{v[\mu_{+}^{\prime}%
+\mu_{-}^{\prime}]}{v+\frac{\lambda}{ik}}dv\right\}  e^{ik\frac{2\pi}{P_{0}}%
x},k\neq0.\nonumber
\end{align}
When $\left\vert k\right\vert >1$, by the Penrose stability condition
(\ref{condition-penrose}), there is no unstable eigenvalue with
$\operatorname{Re}\lambda>0$. So we restrict to the critical case $k=1$.
Define
\begin{equation}
F(w)\equiv\frac{4\pi^{2}}{P_{0}^{2}}\int\frac{v[\mu_{+}^{\prime}+\mu
_{-}^{\prime}]}{v-w}dv. \label{penrose function}%
\end{equation}
Then linear instability is equivalent to that $F(w)=1$ for some $w$ with
$\operatorname{Im}w>0$. We define the Penrose curve: $\{\left(
a,F(a+0i)\right)  \},$ the image of the real axis in the complex plane. For
the flat profiles, the Penrose curve is an interval in the real axis near
$\left(  1,0\right)  $, and under the condition (\ref{condition-positive}) it
passes beyond $(1,0)$ to the right. Now we prove the linear stability of
homogeneous equilibria. We assume the contrary, there is $w_{0}=a+bi$ with
$a\in\mathbf{R},\ b>0$ such that $F(w_{0})=1.$ We now approximate $\mu_{\pm}$
by non-constant $\mu_{\pm}^{n}$ near $0,$ such that $\mu_{\pm}^{n}%
\rightarrow\mu_{\pm}$ in some strong topology, while $\mu_{\pm}^{n\prime
}(0)=0,$ but $\mu_{\pm}^{n\prime}(\xi)<0$ \ \ for $\xi>0.$ We define
\[
F_{n}(w)\equiv\frac{4\pi^{2}}{P_{0}^{2}}\int\frac{v[\mu_{+}^{n\prime}+\mu
_{-}^{n\prime}]}{v-w}dv.
\]
Now the Penrose curves for $F_{n}(w)$ break the constant interval so that
$(1,0)$ is a proper point of the boundary, and from (\ref{condition-penrose}),
$(1,0)$ is the single largest intersection of Penrose curves with real axis.
Since $F(\lambda_{0})=1,$ and $F_{n}(\lambda)$ are analytic near $\lambda
_{0},$ $\lim F_{n}(\lambda)=F(\lambda)$ for $\lambda$ near $\lambda_{0}$ we
therefore conclude that there is $\lambda_{n}$ near $\lambda_{0}$ such that
$F_{n}(\lambda_{n})=1.$ By the open mapping theorem, there is a full
neighborhood $B_{n}$ of $(1,0)$ such that $B_{n}\in F_{n}(\operatorname{Im}%
\lambda>0).\ $This contradicts to the fact that $(1,0)$ is a proper boundary
point. The uniform boundedness of $\ \left(  \mathbf{I}+K(\lambda,0)\right)
^{-1}$ on $\left\{  \operatorname{Re}\lambda>0,\left\vert \lambda\right\vert
\geq\delta\right\}  $ follows from the Penrose condition
(\ref{condition-penrose}) and (\ref{condition-unequal-homo}).

The proof of (2) is by straightforward calculations. To show (3), we note that
$F\left(  0\right)  =F^{\prime}\left(  0\right)  =0$ and by
(\ref{condition-positive})
\begin{equation}
F^{\prime\prime}\left(  0\right)  =\frac{8\pi^{2}}{P_{0}^{2}}\int\frac{\mu
_{+}^{\prime}+\mu_{-}^{\prime}}{v^{2}}dv>0. \label{penrose-second-derivative}%
\end{equation}
Define the 2 by 2 matrix
\[
A\left(  \lambda,0\right)  =\left(  \mathbf{I}+K(\lambda,0)r_{j},r_{l}\right)
,\ j,l=1,2.
\]
Then near $\lambda=0,$
\[
\det A\left(  \lambda,0\right)  =F\left(  \frac{\lambda}{i}\right)  F\left(
-\frac{\lambda}{i}\right)  \backsim\lambda^{4}.
\]

(ii) (\textit{Uneven Case)} The proof of (1) and (2) is similar. Finally, we
note that $F\left(  0\right)  =0,\ F^{\prime}\left(  0\right)  \neq0$ (by
(\ref{condition-nondegenerate})), so $\det A\left(  \lambda,0\right)
\backsim\lambda^{2}$ near $\lambda=0.$
\end{proof}

\begin{proof}
[Proof of Theorem \ref{thm-main-uneven}]We only consider the two species case
since the proof for the fixed ion case is similar. First, we study the
eigenvalues near $0$. Combining Lemma \ref{lemma-spectra-homogeneous} (ii) and
Lemma \ref{lemma-stability-spectra}, we know that the number of eigenvalues
for $\mathbf{I}+K(\lambda,\varepsilon)$ near $0\ $is at most two when
$\varepsilon$ is small. By Lemma \ref{lemma-vlasov-spectra}, the nonzero
eigenvalues must appear in pairs. But clearly the translation invariance leads
to a zero eigenvector $[\mu_{+,\pm}^{\prime}(e_{+})\beta_{x},-\mu_{-,\pm
}^{\prime}(e_{-})\beta_{x},\beta_{x}]$ for the linearized Vlasov-Poisson
equation (\ref{vlasoveigen}) (\ref{poissoneigen}), which gives rise to an zero
eigenvector for $\mathbf{I}+K(0,\varepsilon)$ for all $\varepsilon>0$. So we
conclude that there exists no nonzero eigenvalues near $0$, besides the
translation mode. That is, there exists $\delta_{0}>0$, such that there is no
nonzero eigenvalue $\lambda\ $with $\left\vert \lambda\right\vert \leq
\delta_{0}\ $for $\mathbf{I}+K(\lambda,\varepsilon)$, when $\varepsilon<<1$.
By Lemma \ref{lemma-spectra-away-zero}, when $\varepsilon$ is small enough,
there exists no unstable eigenvalue $\lambda$ of with $\operatorname{Re}%
\lambda>0$ and $\left\vert \lambda\right\vert \geq\delta_{0}$. This proves the
spectral stability of small BGK waves.
\end{proof}

Lastly, we prove Theorem \ref{thm-main-even} for the even case.

\begin{proof}
[Proof of Theorem \ref{thm-main-even}]We give the detailed proof for the two
species case and make some comments on the fixed ion case at the end. Denote
$\mu_{\pm,+}=\mu_{\pm,-}=\mu_{\pm}$. Again by Lemma
\ref{lemma-spectra-away-zero}, it suffices to exclude unstable eigenvalues
near $0$. First, it follows from Lemma \ref{lemma-spectra-homogeneous} (i) and
Lemma \ref{lemma-stability-spectra} that when $\varepsilon<<1,$ there exist at
most four eigenvalues near $\lambda=0$ for $\mathbf{I}+K(\lambda,\varepsilon
)$. By Lemma \ref{lemma-vlasov-spectra}, besides the translation mode, any
nonzero eigenvalue near $0$ must be either purely imaginary or real, since any
eigenvalue with both nonzero real and imaginary parts must appear in quadruple.

Now we exclude real eigenvalues to show spectral stability. Assume $\lambda
\in\mathbf{R}$ and $\left\vert \lambda\right\vert <<1.\ $First, we show that
when $\mu_{\pm}^{\prime}(e_{\pm})$ are even, the operator $\rho(\lambda
,\varepsilon)\phi$ preserves parity. For an even function $\phi\left(
x\right)  \in H_{\varepsilon}^{1}$, let $\phi\left(  x\right)  =\sum
_{k\in\mathbf{Z}}\phi_{k}^{\pm}(I_{\pm})e^{ik\theta_{\pm}}$. Then for
$k\neq0,\ \phi_{k}^{\pm}(I_{\pm})=\phi_{-k}^{\pm}(I_{\pm})$ for untrapped
particles, since $\left(  x,v\right)  $ and $\left(  -x,v\right)  $ have
action-angle coordinates $\left(  I,\theta\right)  $ and $\left(
I,2\pi-\theta\right)  $ respectively. Thus by (\ref{density-lambda-square}),
\begin{align}
\rho(\lambda,\varepsilon)\phi &  =\sum_{k\neq0,\ \pm}\int_{v>0}\frac
{2\omega_{\pm}^{2}\mu_{\pm}^{\prime}(e_{\pm})}{\omega_{\pm}^{2}+\frac
{\lambda^{2}}{k^{2}}}\phi_{k}^{\pm}(I_{\pm})e^{ik\theta_{\pm}}%
dv\label{rhoreal}\\
&  =4\sum_{k>0,\ \pm}\int_{v>0}\frac{\omega_{\pm}^{2}\mu_{\pm}^{\prime}%
(e_{\pm})\phi_{k}^{\pm}(I_{\pm})}{\omega_{\pm}^{2}+\frac{\lambda^{2}}{k^{2}}%
}\cos k\theta_{\pm}dv,\nonumber
\end{align}
is again an even function in $x$. Similarly, it can be shown that when $\phi$
is odd, so is $\rho(\lambda,\varepsilon)\phi$. Thus, we can consider
$\mathbf{I}+K(\lambda,\varepsilon)$ in the even and odd spaces separately. In
the odd subspace $H_{\text{odd }}^{1}$of $H_{0}^{1}$, we notice that
$\mathbf{I}+K(0,0)$ has a 1D kernel spanned by $r_{1}=\sin\frac{2\pi}{P_{0}}%
x$. By (\ref{penrose-operator}) and noticing that for even profiles the
function $F\left(  w\right)  $ defined by (\ref{penrose function}) is even,
so
\begin{align}
\left(  \left(  \mathbf{I}+K(\lambda,0)\right)  r_{1},r_{1}\right)   &
=(\left(  \mathbf{I}+\partial_{x}^{-2}\rho(\lambda,0)\right)  r_{1}%
,r_{1})\label{derivative-penrose-odd}\\
&  =1-F\left(  i\lambda\right)  \backsim a_{0}\lambda^{2}.\nonumber
\end{align}
where $a_{0}=F^{\prime\prime}\left(  0\right)  >0$ (see
(\ref{penrose-second-derivative})) by the condition (\ref{condition-positive}%
). We also note that for any $\varepsilon>0$, $\mathbf{I}+K(0,\varepsilon)$
has a translation kernel which is odd. So by Corollary
\ref{cor-stability-spectra} and Lemma \ref{lemma-vlasov-spectra}, there is no
nonzero eigenvalue $\lambda$ near zero for $\mathbf{I}+K(\lambda,\varepsilon)$
when $\varepsilon<<1$.

We can now restrict ourselves to the even subspace $H_{\text{even }}^{1}$of
$H_{0}^{1}$, so that $\mathbf{I}+K(0,0)$ has 1D kernel spanned by $r_{0}%
=\cos\frac{2\pi}{P_{0}}x$. We note that (\ref{rhoreal}) is valid for all
$\lambda\in\mathbf{C}$ with $|\lambda|<<1$. So letting $\nu=\lambda^{2},$ we
may rewrite $K(\lambda,\varepsilon)$ in terms of $\nu$. That is, we define
\[
\bar{\rho}(\nu,\varepsilon)\phi=4\sum_{k>0,\ \pm}\int_{v>0}\frac{\omega_{\pm
}^{2}\mu_{\pm}^{\prime}(e_{\pm})\phi_{k}^{\pm}(I_{\pm})}{\omega_{\pm}%
^{2}+\frac{\nu}{k^{2}}}\cos k\theta_{\pm}dv,
\]
and
\[
\bar{K}(\nu,\varepsilon)=G_{\beta}\partial_{x}^{-2}\bar{\rho}(\nu
,\varepsilon)G_{\beta}^{-1}.
\]
Then $\bar{K}(\nu,\varepsilon)$ is analytic for $\nu$ near zero. Suppose
$(\mathbf{I}+K(\lambda,\varepsilon))r=0$ or equivalently $(\mathbf{I}+\bar
{K}(\nu,\varepsilon))r=0,$ for $r\in H_{\text{even }}^{1}$and let
$\mathbf{P}r=ar_{0},$ where $\mathbf{P}$ is the projection to $\{r_{0}\}$. As
in the proof of Lemma \ref{lemma-stability-spectra}, we use the
Liapunov-Schmidt reduction to solve $(\mathbf{I}+\bar{K}(\nu,\varepsilon
))r=0\ $to obtain
\[
\left(  I+\bar{K}(\nu,0)\right)  r+\left(  \bar{K}(\nu,\varepsilon)-\bar
{K}(\nu,0)\right)  \{Z^{\perp}(\nu,\varepsilon)+\mathbf{I}\}\mathbf{P}r=0,
\]
which is equivalent to
\begin{align}
0 &  =\left(  \left(  \mathbf{I}+\bar{K}(\nu,0)\right)  r_{0},r_{0}\right)
\label{implicit}\\
&  \ +\left(  \left(  \bar{K}(\nu,\varepsilon)-\bar{K}(\nu,0)\right)  \left[
Z^{\perp}(\nu,\varepsilon)+\mathbf{I}\right]  r_{0},r_{0}\right)  ,\nonumber
\end{align}
where
\begin{align*}
Z^{\perp}(\nu,\varepsilon) &  \equiv-[\mathbf{I}+\left(  \mathbf{I}+\bar
{K}(\nu,0)\right)  ^{-1}\left(  \mathbf{I}-\mathbf{P}\right)  \left(  \bar
{K}(\nu,\varepsilon)-\bar{K}(\nu,0)\right)  ]^{-1}\\
&  \left(  \mathbf{I}+\bar{K}(\lambda,0)\right)  ^{-1}\left(  \mathbf{I}%
-\mathbf{P}\right)  [\bar{K}(\nu,\varepsilon)-\bar{K}(\nu,0)]
\end{align*}
as defined in (\ref{definition-Z-per}) is analytic in $\nu$ near $0$. Since
the eigenvalue $\lambda$ must be real or pure imaginary, so it suffices to
study the equation (\ref{implicit}) for $\nu=\lambda^{2}\in\mathbf{R}$. As in
(\ref{derivative-penrose-odd}), we have
\begin{align}
\left(  \left(  \mathbf{I}+\bar{K}(\nu,0)\right)  r_{0},r_{0}\right)   &
=\left(  \left(  \mathbf{I}+\bar{K}(\nu,0)\right)  r_{0},r_{0}\right)
\label{derivative-penrose-even}\\
&  =1-F\left(  i\lambda\right)  \backsim a_{0}\lambda^{2}=a_{0}\nu\text{,
}\nonumber
\end{align}
with $a_{0}>0$. So from the implicit function theorem, for $\varepsilon<<1$
there is an unique $\nu(\varepsilon)\in\mathbf{R\ }$so that (\ref{implicit})
is valid, with $\nu(0)=0$. There is a growing mode $\lambda\left(
\varepsilon\right)  >0$ \textit{if and only if }$\nu(\varepsilon
)=\lambda\left(  \varepsilon\right)  ^{2}>0,$ and there is a purely imaginary
eigenvalue $0\neq\lambda\left(  \varepsilon\right)  \in i\mathbf{R}$
\textit{if and only if }$\nu(\varepsilon)=\lambda\left(  \varepsilon\right)
^{2}<0$. By (\ref{derivative-penrose-even}) and (\ref{implicit}) we need
\[
S(\nu,\varepsilon)\equiv\left(  \left(  \bar{K}(\nu,\varepsilon)-\bar{K}%
(\nu,0)\right)  \left(  Z^{\perp}(\nu,\varepsilon)+\mathbf{I}\right)
r_{0},r_{0}\right)  <0
\]
for $\varepsilon<<1$ to ensure that $\nu(\varepsilon)>0$. Note that
\[
S(\nu,\varepsilon)=[S(\nu,\varepsilon)-S(0,\varepsilon)]+S(0,\varepsilon)
\]
and
\begin{equation}
\left\vert \lbrack S(\nu,\varepsilon)-S(0,\varepsilon)]\right\vert \leq
\max_{v^{\prime}\in\left(  0,\nu\right)  }\left\vert \partial_{\nu}%
S(\nu^{\prime},\varepsilon)\right\vert \left\vert \nu\right\vert
,\label{difference-s}%
\end{equation}
where
\begin{align*}
\partial_{\nu}S(\nu,\varepsilon) &  =\left(  \partial_{\nu}\left[  \bar{K}%
(\nu,\varepsilon)-\bar{K}(\nu,0)\right]  \left[  Z^{\perp}(\nu,\varepsilon
)+\mathbf{I}\right]  r_{0},r_{0}\right)  \\
&  +\left(  \left[  \bar{K}(\nu,\varepsilon)-\bar{K}(\nu,0)\right]
\partial_{\nu}Z^{\perp}(\nu,\varepsilon)r_{0},r_{0}\right)  .
\end{align*}
For $\nu=\lambda^{2}$ near $0$, by lemma \ref{realderivative},
\begin{align*}
&  \ \ \ \ \ \ \left\vert \left(  \partial_{\nu}\left(  \bar{K}(\nu
,\varepsilon)-\bar{K}(\nu,0)\right)  \left(  Z^{\perp}(\nu,\varepsilon
)+\mathbf{I}\right)  r_{0},r_{0}\right)  \right\vert \\
&  =\frac{1}{2\left\vert \lambda\right\vert }\left\vert \left(  \partial
_{\lambda}\left(  K(\lambda,\varepsilon)-K(\lambda,0)\right)  \left(
Z^{\perp}(\nu,\varepsilon)+\mathbf{I}\right)  r_{0},r_{0}\right)  \right\vert
\\
&  \lesssim\varepsilon\left\Vert \left(  Z^{\perp}(\nu,\varepsilon
)+\mathbf{I}\right)  r_{0}\right\Vert _{W^{1,\infty}}\left\Vert r_{0}%
\right\Vert _{L^{\infty}}\lesssim\varepsilon,
\end{align*}
and by Proposition \ref{prop-operator-difference},
\begin{align*}
\   & \left\vert \left(  \left(  \bar{K}(\nu,\varepsilon)-\bar{K}%
(\nu,0)\right)  \partial_{\nu}Z^{\perp}(\nu,\varepsilon)r_{0},r_{0}\right)
\right\vert \\
& =\left\vert \left(  \left(  K(\lambda,\varepsilon)-K(\lambda,0)\right)
\partial_{\nu}Z^{\perp}(\nu,\varepsilon)r_{0},r_{0}\right)  \right\vert
\lesssim\sqrt{\varepsilon}.
\end{align*}
Noting that $\nu(\varepsilon)=\nu^{\prime}(0)\varepsilon+\frac{\nu
^{\prime\prime}(0)\varepsilon^{2}}{2!}+\cdots=O\left(  \varepsilon\right)  $,
thus a combination of above estimates and (\ref{difference-s}) gives
\[
\left\vert S(\nu,\varepsilon)-S(0,\varepsilon)\right\vert =O\left(
\varepsilon^{\frac{3}{2}}\right)  .
\]
Since $S(0,0)=0,$ so a sharp \textit{stability criterion} is for the first
nonzero derivative:
\[
\frac{d}{d\varepsilon}S(0,\varepsilon)|_{\varepsilon=0}<0\ \ \ \text{for
instability,}%
\]
and
\[
\frac{d}{d\varepsilon}S(0,\varepsilon)|_{\varepsilon=0}>0\ \text{\ for
stability.}%
\]
Moreover, under the stability condition $\frac{d}{d\varepsilon}S(0,\varepsilon
)|_{\varepsilon=0}>0$, from our proof we have a pair of imaginary eigenvalues
bifurcating from $\varepsilon=0.$ Note that since $Z^{\perp}(0,0)=0,$%
\[
\frac{d}{d\varepsilon}S(0,\varepsilon)|_{\varepsilon=0}=\frac{d}{d\varepsilon
}(G_{\beta}\partial_{x}^{-2}\rho(0,\varepsilon)G_{\beta}^{-1}r_{0}%
,r_{0})|_{\varepsilon=0}.
\]
For $r_{0}=\cos(\frac{2\pi}{P_{0}}x),$%
\[
\partial_{x}^{-2}G_{\beta}^{-1}r_{0}=\partial_{x}^{-2}\cos(\frac{P_{0}%
}{P_{\beta}}\frac{2\pi}{P_{0}}x)=-\frac{P_{\beta}^{2}}{4\pi^{2}}\cos
(\frac{2\pi}{P_{\beta}}x).
\]
Notice that for $\phi\in H_{\varepsilon}^{1},$
\[
\rho(0,\varepsilon)\phi=\int_{\mathbf{R}}\left(  \mu_{+}^{\prime}(e_{+}%
)+\mu_{-}^{\prime}(e_{-})\right)  dv\ \phi.
\]
So
\begin{align*}
S(0,\varepsilon) &  =\frac{P_{0}}{P_{\beta}}(\partial_{x}^{-2}\rho
(0,\varepsilon)G_{\beta}^{-1}r_{0},G_{\beta}^{-1}r_{0})\\
&  =-\frac{P_{0}P_{\beta}}{4\pi^{2}}\int_{0}^{P_{\beta}}\int\left(  \mu
_{+}^{\prime}(e_{+})+\mu_{-}^{\prime}(e_{-})\right)  dv\ \cos^{2}(\frac{2\pi
}{P_{\beta}}x)dx\\
&  =-\frac{P_{\beta}^{2}}{4\pi^{2}}\int_{0}^{P_{0}}\int[\mu_{+}^{\prime
}\left(  \frac{1}{2}v^{2}+\beta\left(  \frac{P_{\beta}}{P_{0}}x\right)
\right)  \\
&  \ \ \ \ \ +\mu_{-}^{\prime}\left(  \frac{1}{2}v^{2}-\beta\left(
\frac{P_{\beta}}{P_{0}}x\right)  \right)  ]\ dv\ \cos^{2}(\frac{2\pi}{P_{0}%
}x)dx.
\end{align*}
Notice that by (\ref{leading-order-beta}),
\[
\frac{d}{d\varepsilon}\beta\left(  \frac{P_{\beta}}{P_{0}}x\right)
|_{\varepsilon=0}=\cos\left(  \frac{2\pi}{P_{0}}x\right)  .
\]
So $\ $
\begin{align*}
&  \ \ \ \ \ \frac{d}{d\varepsilon}S(0,\varepsilon)|_{\varepsilon=0}\\
&  =-\frac{P_{0}P_{\beta}^{^{\prime}}\left(  0\right)  }{2\pi^{2}}\int%
_{0}^{P_{0}}\int\left(  \mu_{+}^{\prime}\left(  \frac{1}{2}v^{2}\right)
+\mu_{-}^{\prime}\left(  \frac{1}{2}v^{2}\right)  \right)  dv\cos^{2}%
(\frac{2\pi}{P_{0}}x)dx\\
&  \ \ \ \ \ \ \ -\frac{P_{0}^{2}}{4\pi^{2}}\int_{0}^{P_{0}}\int\left(
\mu_{+}^{\prime\prime}\left(  \frac{1}{2}v^{2}\right)  -\mu_{-}^{\prime}%
(\frac{1}{2}v^{2}\right)  dv\cos^{3}(\frac{2\pi}{P_{0}}x)dx\\
&  =-\frac{P_{0}^{2}P_{\beta}^{^{\prime}}\left(  0\right)  }{4\pi^{2}}%
\int\left(  \mu_{+}^{\prime}\left(  \frac{1}{2}v^{2}\right)  +\mu_{-}^{\prime
}\left(  \frac{1}{2}v^{2}\right)  \right)  dv=-P_{\beta}^{^{\prime}}\left(
0\right)  .
\end{align*}
So we deduce the stability criterion that $P_{\beta}^{\prime}\left(  0\right)
<0$ for stability and $P_{\beta}^{\prime}\left(  0\right)  >0$ for
instability. Moreover, in the stable case $\left(  P_{\beta}^{\prime
}>0\right)  $, there exists a pair of nonzero purely imaginary eigenvalues.

Lastly, we calculate $P_{\beta}^{\prime}\left(  0\right)  $. Recall
\[
-\beta_{xx}=h(\beta)=\int\mu_{+}(\frac{v^{2}}{2}+\beta)-\int\mu_{-}%
(\frac{v^{2}}{2}-\beta)=H^{\prime}(\beta).
\]
Since $H^{\prime}(0)=0,$ $H^{\prime\prime}(0)=(\frac{2\pi}{P_{0}})^{2},$
\[
H^{\prime\prime\prime}(0)=\int\mu_{+}^{\prime\prime}(\frac{v^{2}}{2})-\int%
\mu_{-}^{\prime\prime}(\frac{v^{2}}{2})=\int\frac{\mu_{+}^{^{\prime}}%
(\frac{v^{2}}{2})}{v^{2}}-\int\frac{\mu_{-}^{\prime}(\frac{v^{2}}{2})}{v^{2}%
}.
\]
Let $a_{2}=(\frac{2\pi}{P_{0}})^{2},\ a_{3}=H^{\prime\prime\prime}(0)$, and
define $g\left(  u\right)  =2H\left(  u\right)  -uh\left(  u\right)  $, then
\begin{align*}
P_{\beta}^{\prime}\ |_{\beta=0}  &  =\lim_{\beta\rightarrow0}\frac{4}{\beta
}\int_{0}^{\beta}\frac{g\left(  \beta\right)  -g\left(  u\right)  }{\left(
\sqrt{2\left(  H(\beta)-H(u)\right)  }\right)  ^{3}}du\\
&  =\lim_{\beta\rightarrow0}\frac{4}{\beta}\frac{1}{6}a_{3}\int_{0}^{\beta
}\frac{\beta^{3}-u^{3}}{\left(  a_{2}\left(  \beta^{2}-u^{2}\right)  \right)
^{\frac{3}{2}}}du\\
&  =\frac{1}{6}a_{3}\int_{0}^{1}\frac{1-u^{3}}{\left(  a_{2}\left(
1-u^{2}\right)  \right)  ^{\frac{3}{2}}}du.
\end{align*}
So we deduce instability when
\[
\int\frac{\mu_{+}^{\prime}(\frac{v^{2}}{2})}{v^{2}}-\int\frac{\mu_{-}^{\prime
}(\frac{v^{2}}{2})}{v^{2}}>0\ \left(  \text{equivalently }P_{\beta}^{\prime
}>0\right)  ,
\]
$\ $and stability when
\[
\int\frac{\mu_{+}^{\prime}(\frac{v^{2}}{2})}{v^{2}}-\int\frac{\mu_{-}^{\prime
}(\frac{v^{2}}{2})}{v^{2}}<0\ \ \left(  \text{equivalently }P_{\beta}^{\prime
}<0\right)  .
\]

The proof for the fixed ion case is similar. Here, the condition
(\ref{condition-positive-fixed ion}) plays a duel role. It is the stability
condition for the homogeneous state $\left(  \mu\left(  \frac{1}{2}%
v^{2}\right)  ,0\right)  $ at the critical period $P_{0}$, and is also
equivalent to the instability condition $P_{\beta}^{\prime}>0$ for small BGK waves.
\end{proof}

\section{Examples}

1. The flat profiles satisfying conditions in Theorem \ref{thm-main-uneven}
and \ref{thm-main-even} can be easily constructed. We give an example for the
uneven and one species case. Let
\[
f\left(  v\right)  =c_{0}\left(  e^{-\left(  v-v_{1}\right)  ^{2}}+e^{-\left(
v+v_{1}\right)  ^{2}}\right)  ,
\]
where $v_{1}>0$ and $c_{0}>0$ is such that $\int f\left(  v\right)  dv=1$.
Note that $f\left(  v\right)  $ has two maximum at $v_{1}$ and $-v_{1}$ and
one minimum at $0$. If $v_{1}$ is large enough, then
\[
\int\frac{f^{\prime}\left(  v\right)  }{v\pm v_{1}}dv<\int\frac{f^{\prime
}\left(  v\right)  }{v}dv\text{ and }\int\frac{f^{\prime}\left(  v\right)
}{v}dv>0.
\]
We now modify $f\left(  v\right)  $ slightly to get $f_{0}\left(  v\right)  $
which is flat in a small interval $\left[  -\sigma,\sigma\right]  $ and
slightly uneven such that $\int\frac{f_{0}^{\prime}\left(  v\right)  }{v^{2}%
}dv\neq0$. Let $v_{a}$ and $v_{b}$ be the two maximum points of $f_{0}\left(
v\right)  $, then we still have
\[
\max\left\{  \int\frac{f_{0}^{\prime}\left(  v\right)  }{v-v_{a}}dv,\int%
\frac{f_{0}^{\prime}\left(  v\right)  }{v-v_{b}}dv\right\}  <\int\frac
{f_{0}^{\prime}\left(  v\right)  }{v}dv=\left(  \frac{2\pi}{P_{0}}\right)
^{2}.
\]
Since $\int\frac{f_{0}^{\prime}\left(  v\right)  }{v^{2}}dv\neq0$, when
$\sigma$ is small, the function $F\left(  w\right)  =\int\frac{f_{0}^{\prime
}\left(  v\right)  }{v-w}dv$ is monotone in $\left[  -\sigma,\sigma\right]  $
and the condition (\ref{condition-unequal-homo-fixed}) is also satisfied.
Thus, the profiles $f_{0}\left(  v\right)  $ satisfies all the conditions in
Theorem \ref{thm-main-uneven}. The small BGK waves bifurcating near $\left(
f_{0}\left(  v\right)  ,0\right)  $ are spectrally stable.

2. For the even case, the condition $\int\frac{[\mu_{+}^{\prime}+\mu
_{-}^{\prime}]}{v^{2}}dv>0$ is always true when the flatness width
$\sigma_{\pm}$ are small enough. Consider $\mu_{\pm}$ to be the modification
(flattening near $0$) of non-flat profiles with local minimum at $0$. Let
$\sigma_{0}=\max\left\{  \sigma_{+},\sigma_{-}\right\}  $ and take an interval
$I=\left[  2\sigma_{0},a\right]  $ such that $\mu_{\pm}^{\prime}\geq c_{0}>0$
in $I$. Then
\[
\int_{I}\frac{[\mu_{+}^{\prime}+\mu_{-}^{\prime}]}{v^{2}}dv\geq2c_{0}\int%
_{I}\frac{1}{v^{2}}=2c_{0}\left(  \frac{1}{2\sigma_{0}}-\frac{1}{a}\right)  .
\]
Since $\mu_{\pm}^{\prime}\geq0$ in $\left[  0,2\sigma_{0}\right]  $,
\begin{align*}
\int\frac{[\mu_{+}^{\prime}+\mu_{-}^{\prime}]}{v^{2}}dv  &  =2\int_{v>0}%
\frac{[\mu_{+}^{\prime}+\mu_{-}^{\prime}]}{v^{2}}dv\\
&  \geq c_{0}\left(  \frac{1}{2\sigma_{0}}-\frac{1}{a}\right)  +\int%
_{v>a}\frac{[\mu_{+}^{\prime}+\mu_{-}^{\prime}]}{v^{2}}dv\\
&  >0,
\end{align*}
when $\sigma_{0}$ is small enough. By the same argument, both stability and
instability conditions
\begin{equation}
\int v^{-2}[\mu_{+}^{\prime}(\frac{v^{2}}{2})-\mu_{-}^{\prime}(\frac{v^{2}}%
{2})]>0\,\ \left(  <0\right)  \label{sign}%
\end{equation}
can be satisfied by choosing different $\sigma_{\pm}$ for $\mu_{\pm}$. Indeed,
when $1\gg\sigma_{-}\gg\sigma_{+}$ ($\sigma_{-}\ll\sigma_{+}\ll1$), we get
$+\left(  -\right)  $ sign in (\ref{sign}).

3. For the one species and even case, by the same arguments as above, the
instability condition $\int\frac{\mu^{\prime}\left(  \frac{1}{2}v^{2}\right)
}{v^{2}}dv>0$ is always satisfied when the flatness width is small. So BGK
waves near $\mu\left(  \frac{1}{2}v^{2}\right)  $ are linearly unstable. But
for a slightly uneven profile $f_{0}\left(  v\right)  $ close to $\mu\left(
\frac{1}{2}v^{2}\right)  $, by Theorem \ref{thm-main-uneven}, the BGK waves
near $f_{0}\left(  v\right)  $ are linearly stable. These suggest that there
exist a transition from stability to instability when we increase the
amplitude of BGK waves bifurcated from $f_{0}\left(  v\right)  $ up to the one
close to an unstable BGK wave bifurcated from $\mu\left(  \frac{1}{2}%
v^{2}\right)  $.

4. In Theorems \ref{thm-main-even} and \ref{thm-main-uneven}, we construct
linearly stable BGK waves for both even and uneven cases. However, there is a
significant difference in their spectra which must lie in the imaginary axis.
For the uneven case, there is no nonzero imaginary eigenvalue of the
linearized VP operator at a stable small BGK wave. In contrast, for the even
case, there exists a pair of nonzero imaginary eigenvalues for the stable
small BGK waves. In particular, these imply that there is no linear damping
for the even BGK waves due to the existence of nonzero time periodic solutions
of linearized Vlasov-Poisson equation. For the uneven case, the linear damping
is under investigation (\cite{guo-lin-damping}).

\begin{center}
{\Large Acknowledgement}
\end{center}

Yan Guo's research is supported in part by NSF grant DMS-1209437, Chinese NSF
grant \#10828103 and a Simon Research Fellowship. Zhiwu Lin is supported in
part by a NSF grant DMS-1411803. The authors thank referees for their
extensive comments which help to improve the presentation of the paper.

\end{document}